\documentclass[11pt]{article}

\usepackage{amsmath, amssymb,mathtools,tikz}

\usepackage{hyperref}
  \usepackage{color}
\usetikzlibrary{shapes.arrows}
\usetikzlibrary{arrows,decorations.markings}

\usepackage{chngcntr}

\counterwithin{figure}{section}
\counterwithin{table}{section}

\include{epic.sty}

\newlength{\originalbase}
\newcommand{\spacing}[1]{\setlength{\baselineskip}{#1\originalbase}}

\begin{document}
\spacing{1.2}
\parskip=+1pt

\renewcommand{\thefigure}{\thesection.\arabic{figure}}
\renewcommand{\thetable}{\thesection.\arabic{table}}

  \newenvironment{proof}{\vspace{1ex}\noindent{\bf Proof:}}{\hspace*{\fill}
  $\blacksquare$\vspace{1ex}}
  \newenvironment{proofof}[1]{\vspace{1ex}\noindent{\bf Proof of #1:}}{\hspace*{\fill}
  $\blacksquare$\vspace{1ex}}
  \def\noproof{\hspace*{\fill}$\blacksquare$}

  \newtheorem{theorem}{Theorem}[section]
  \newtheorem{lemma} [theorem] {Lemma}
  \newtheorem{corollary} [theorem] {Corollary}
  \newtheorem{claim} [theorem] {Claim}
  \newtheorem{proposition} [theorem] {Proposition}
  \newtheorem{definition} [theorem] {Definition}
  \newtheorem{remark} [theorem] {Remark}
  \newtheorem{conjecture} [theorem] {Conjecture}
\newtheorem{question} [theorem] {Question}

\newtheorem{example} [theorem] {Example}

\def\Pr{\mbox{Pr}}

\def\Tsi{T_{\rm si}}
\def\Tfor{T_{\rm forget}}
\def\Tmix{T_{\rm mix}}
\def\Tres{T_{\rm reset}}
\def\Tbest{T_{\rm bestmix}}

\def\Ret{\mbox{Ret}}

\def\V{V}
\def\Vc{\overline{V}}
\def\W{W}
\def\Wc{\overline{W}}
\def\Sc{\overline{S}}
\def\Vu{V_{u:v}}
\def\Vv{V_{v:u}}
\def\Vi{V_i}
\def\Va{V_{a:b}}
\def\Vb{V_{b:a}}
\def\Vuc{\overline{V_u}}
\def\Vvc{\overline{V_v}}
\def\Vic{\overline{V_i}}


\def\tbm{T_{\rm{bestmix}}}
\def\tm{T_{\rm{mix}}}
\def\tf{T_{\rm{forget}}}
\def\tres{T_{\rm{reset}}}

\def\wb{Y}

\def\T{\mathcal{T}}
\def\G{\mathcal{G}}
\def\nG{\widetilde{G}}
\def\nE{\widetilde{E}}

\def\D{\mathcal{D}}
\def\S{\mathcal{S}}

\def\npi{\widetilde{\pi}}
\def\pdelt{\Delta\mathrm{Pess}}
\def\hp{\mathrm{Pess}}
\newcommand{\p}[1]{#1'}
\newcommand{\psq}[1]{#1''}

\newcommand{\tr}[2]{\tau(#1,#2)}
\newcommand{\str}[2]{\sigma(#1,#2)}

\def\R{R}
\def\L{L}

\def\lpess{\mathrm{left}(G)}
\def\nlpess{\mathrm{left}(\nG)}
\def\rpess{\mathrm{right}(G)}
\def\nrpess{\mathrm{right}(\nG)}

\def\qed{\hfill$\Box$}

\newcommand{\case}[1]{{\bf Case #1:}}

\def\blue{\color{blue}}
\def\grn{\color{green}}
\def\cyan{\color{cyan}}
\def\magenta{\color{magenta}}
\def\brown{\color{brown}}
\def\black{\color{black}}
\def\yellow{\color{yellow}}
\def\grey{\color{grey}}

\newcommand{\bnote}[1]{\marginpar{{\blue \footnotesize\raggedright#1}}}
\newcommand{\gnote}[1]{\marginpar{{\grn \footnotesize\raggedright#1}}}
\newcommand{\rnote}[1]{\marginpar{{\red \footnotesize\raggedright#1}}}

\newif\ifnotesw\noteswtrue
\newcommand{\mycomment}[1]{\ifnotesw $\blacktriangleright$\ {\sf #1}\ 
  $\blacktriangleleft$ \fi}

\newcommand{\bcomment}[1]{{\blue \mycomment{#1}}}


\title{The best mixing time for random walks on trees}
\author{Andrew Beveridge\footnote{
Department of Mathematics, Statistics and Computer Science, Macalester College, Saint Paul, MN 55105}
~and Jeanmarie Youngblood\footnote{Department of Mathematics, University of Minnesota, Minneapolis, MN 55455}}


\maketitle

\begin{abstract}
We characterize the extremal structures for mixing walks on trees that start from the most advantageous vertex. 
Let $G=(V,E)$ be a tree with stationary distribution $\pi$. For a vertex $v \in V$, let $H(v,\pi)$ denote the expected length of an optimal stopping rule from $v$ to $\pi$. The \emph{best mixing time} for $G$ is $\min_{v \in V} H(v,\pi)$. We show that among all trees with $|V|=n$, the best mixing time is minimized uniquely by the star. For even $n$, the best mixing time is maximized by the uniquely path. Surprising, for odd $n$, the best mixing time is maximized uniquely by a path of length $n-1$ with a single leaf adjacent to one central vertex.
\end{abstract}

\section{Introduction}

We resolve the following extremal question for random walks on trees: what tree structure minimizes/maximizes the expected length of an optimal  stopping rule to the stationary distribution $\pi$, given that we start at the most advantageous vertex? Naturally, the star $S_n$ is the minimizing tree structure, but the maximization problem has an unexpected twist. The path $P_{n}$ is the maximizing structure when $n$ is even, but for odd $n$ the maximizing structure is the near-path  $Y_n$ consisting of a path on $n-1$ vertices with a single leaf adjacent to one of the two central vertices. We refer to this graph $Y_n$ as the \emph{wishbone}. This choice of name is suggested by the layout of $Y_n$ in Figure \ref{fig:menagerie}(c) below. 

A \emph{random walk} on an undirected graph $G=(V,E)$ consists of a sequence of vertices $(w_0, w_1, \ldots , w_n, \ldots)$ such that 
$\Pr[w_{t+1} = w \mid w_t = v]$ is $1/ \deg(v)$ if $(v,w) \in E$ and $0$ otherwise. The \emph{hitting time} $H(v,w)$ from vertex $v$ to vertex $w$  is the expected number of steps before a random walk started at $v$ visits $w$ for the first time. We also define $H(v,v)=0$, while $\Ret(v)$ denotes the expected number of steps before a random walk started at $v$ first returns to $v$.

When $G$ is not bipartite, the distribution of $w_t$ converges to the \emph{stationary distribution} $\pi$, where 
$\pi_v = \deg(v) /2|E|.$ 
Inconveniently, we do not have convergence for bipartite graphs (including trees), but we can rectify this at the cost of doubling the expected length of any walk.
Indeed the distribution $w_t$ converges to $\pi$ when we follow  a \emph{lazy random walk} in which we remain at the current state with probability $1/2$ at each time step. 
A \emph{mixing measure} of a graph $G$ captures the rate of convergence to the stationary distribution $\pi$. Let $\sigma=\sigma_0$ denote our initial distribution and $\sigma_t$ denote the distribution for the $t$-th step of a random walk. For a fixed constant $0 < \epsilon < 1$, the (approximate) mixing time of  $G$ is given by $\Tmix(\epsilon) = \max_{\sigma} \min \{T \geq 0   : \, \| \sigma_t - \pi \| < \epsilon \mbox{ for all } t \geq T \},$ where  the maximum is taken over all possible starting distributions and $\| \cdot \|$ is the given metric. This mixing time depends upon the choice of the parameter $\epsilon$. 

Lov\'asz and Winkler \cite{lovasz+winkler} studied a class of parameterless mixing measures by using more sophisticated \emph{stopping rules} to drive the random walk to a desired distribution. Suppose that we are given a \emph{starting distribution} $\sigma$ and a \emph{target distribution} $\tau$. A \emph{$(\sigma,\tau)$ stopping rule} halts a random walk whose initial state is drawn from $\sigma$ so that the final state is governed by the distribution $\tau$. The \emph{access time} $H(\sigma, \tau)$ denotes the minimum expected length for a such a $(\sigma,\tau)$ stopping rule to halt. We say that a stopping rule is \emph{optimal} if it achieves this minimum expected length. Using access times, we have three natural mixing measures, the \emph{mixing time}, the \emph{reset time}, and the \emph{best mixing time} given respectively by
\[
\Tmix = \max_{v \in V} H(v, \pi), \quad \mbox{and} \quad \Tres = \sum_{v \in V} \pi_v H(v, \pi),
 \quad \mbox{and} \quad \Tbest = \min_{v \in V} H(v, \pi).
\]
These are the  worst-case, average-case and best-case mixing measures. These  quantities are called \emph{exact mixing measures}, as opposed to approximate mixing measures like $\Tmix(\epsilon)$ above. See \cite{ALW} for a taxonomy that compares exact and approximate mixing measures for Markov chains. We also note that exact stopping rules converge to $\pi$ on a bipartite graph, even for non-lazy walks. Indeed, stopping rules can employ  randomness in deciding when to stop, which has the same periodicity-breaking effect as using a lazy walk.

It is natural to wonder which graph structures lead to slow or rapid mixing, and the study of extremal graphs for random walks is well-established, cf. \cite{aldous+fill,bright, feige}. A natural line of inquiry is to restrict our attention to trees. The extremal tree structures on $n$ vertices for $\Tmix$ and $\Tfor$ were characterized in \cite{beveridge+wang}. Not surprisingly, the unique minimal structure is the star $S_n=K_{1,n-1}$ and the unique maximal structure is the path.

\begin{theorem}[\cite{beveridge+wang}]
\label{thm:mix-reset}
If $G$ is a tree on $n \geq 3$ vertices then
\[
\frac{3}{2} \leq \Tmix \leq \frac{2n^2 - 4n + 3}{6}
\]
and
\[
1 \leq \Tres \leq \left\{
\begin{array}{cc}
\frac{1}{4}(n^2 - 2n +2) & \mbox{if $n$ is even}, \\
\frac{1}{4}(n-1)^2 & \mbox{if $n$ is odd}. \\
\end{array}
\right.
\]
In each instance, the lower bound is achieved uniquely by the star $S_n$ and the upper bound is achieved uniquely by the path $P_n$.
\end{theorem}

We add to these extremal tree characterizations by studying the best mixing time $\Tbest$.

\begin{theorem}
\label{thm:bestmix}
Let $G$ be a tree on $n \geq 3$ vertices. Then
\[
1/2 \leq \Tbest(G) \leq 
\left\{
\begin{array}{rl}
\frac{1}{12}(n^2 +4n -6) & \mbox{if $n$ is even}, \\
\frac{1}{12}(n^2+4n-15) & \mbox{if $n$ is odd}. \\
\end{array}
\right.
\]
The star $S_n$ is the unique minimizing tree.
For $n$ even, the path $P_n$ is the unique maximizing tree. For $n$ odd, the wishbone $Y_n$ is the unique
maximizing tree.
\end{theorem}
Below, we assume that  $n \geq 4$ since there is a unique tree for $n=3$ (and the formula is correct in that case). The appearance of the wishbone is quite unexpected. At the end of the next section, we discuss the characteristics of $Y_n$ that lead to its maximization of the best mixing time for odd $n$. 

The proof of Theorem \ref{thm:bestmix} is more involved than the proof of  Theorem \ref{thm:mix-reset} in \cite{beveridge+wang}. The reason is that a minor change to the tree can abruptly shift the location of the initial vertex achieving $\tbm$, whereas the behavior of $\Tmix$ and $\Tres$ are more stable under minor structural alterations. Here is an overview of our proof. First, we  replace the given  tree with a caterpillar having a larger best mixing time. We then follow a prescribed algorithm, carefully moving one or two leaves at a time via a \emph{tree surgery} $S_i$. These surgeries monotonically increase $\Tbest$, until we arrive at an even path or an odd wishbone. Figure \ref{fig:examples} shows two sequences of surgeries, the first resulting in $P_{10}$ and the second resulting in $Y_{11}$. The surgeries and the guiding algorithm are described in Section \ref{sec:proof}. 

\begin{figure}[t]
\begin{center}
\begin{tikzpicture}[scale=.5, every node/.style={font=\scriptsize},
    decoration={
      markings,
      mark=at position 1 with {\arrow[scale=1.25,black]{latex}};
    }
  ]

%
%
\begin{scope}

\draw (1,0) -- (7,0);

\draw (2,0) -- (2,-1);
\draw (5,0) -- (4.67,-1);
\draw (5,0) -- (5.33,-1);

\foreach \i in {1,2,3,5,6,7}
{
\draw[fill] (\i,0) circle (4pt);
}

\draw[fill=black!20] (2,-1) circle (4pt);
\draw[fill=black!40] (4.67,-1) circle (4pt);
\draw[fill=black!60] (5.33,-1) circle (4pt);

\draw[fill=white] (4,0) circle (6pt);
\draw[fill=white] (4,0) circle (4pt);

\draw[bend left=20, postaction=decorate] (8, -.5) to (8,-2.5);
\node at (8.75,-1.5) {$S_1$};

\node at (2,-1.6) {$a$};
\node at (4.67,-1.5) {$b$};
\node at (5.33,-1.6) {$c$};
\node at (7,-.5) {$d$};

\end{scope}

\begin{scope}[shift={(0,-3)}]

\draw (1,0) -- (6,0);

\draw (2,0) -- (2,-1);
\draw (4,0) -- (4,-1);
\draw (5,0) -- (4.67,-1);
\draw (5,0) -- (5.33,-1);

\foreach \i in {1,2,5,6}
{
\draw[fill] (\i,0) circle (4pt);
}

\draw[fill=black!20] (2,-1) circle (4pt);
\draw[fill=black] (4,-1) circle (4pt);
\draw[fill=black!40] (4.67,-1) circle (4pt);
\draw[fill=black!60] (5.33,-1) circle (4pt);

\draw[fill=white] (3,0) circle (4pt);

\draw[fill=white] (4,0) circle (6pt);
\draw[fill=white] (4,0) circle (4pt);

\draw[bend left=20, postaction=decorate] (8, -.5) to (8,-2.5);
\node at (8.75,-1.5) {$S_4$};

\node at (2,-1.6) {$a$};
\node at (4.67,-1.5) {$b$};
\node at (5.33,-1.6) {$c$};
\node at (4,-1.5) {$d$};

\end{scope}

\begin{scope}[shift={(0,-6)}]

\draw (1,0) -- (7,0);

\draw (2,0) -- (2,-1);
\draw (4,0) -- (3.67,-1);
\draw (4,0) -- (4.33,-1);

\foreach \i in {1,2,5,6}
{
\draw[fill] (\i,0) circle (4pt);
}

\draw[fill=black!20] (2,-1) circle (4pt);
\draw[fill=black] (3.67,-1) circle (4pt);
\draw[fill=black!40] (4.33,-1) circle (4pt);
\draw[fill=black!60] (7,0) circle (4pt);

\draw[fill=white] (3,0) circle (4pt);

\draw[fill=white] (4,0) circle (6pt);
\draw[fill=white] (4,0) circle (4pt);

\draw[bend left=20, postaction=decorate] (8, -.5) to (8,-2.5);
\node at (8.75,-1.5) {$S_6$};

\node at (2,-1.6) {$a$};
\node at (4.33,-1.5) {$b$};
\node at (7,-.6) {$c$};
\node at (3.67,-1.5) {$d$};

\end{scope}

\begin{scope}[shift={(0,-9)}]

\draw (0,0) -- (7,0);

\draw (4,0) -- (3.67,-1);
\draw (4,0) -- (4.33,-1);

\foreach \i in {1,2,5,6}
{
\draw[fill] (\i,0) circle (4pt);
}

\draw[fill=black!20] (0,0) circle (4pt);
\draw[fill=black] (3.67,-1) circle (4pt);
\draw[fill=black!40] (4.33,-1) circle (4pt);
\draw[fill=black!60] (7,0) circle (4pt);

\draw[fill=white] (3,0) circle (6pt);
\draw[fill=white] (3,0) circle (4pt);

\draw[fill=white] (4,0) circle (6pt);
\draw[fill=white] (4,0) circle (4pt);

\draw[bend left=20, postaction=decorate] (8, -.5) to (8,-2.5);
\node at (8.85,-1.5) {$S_{11}$};

\node at (0,-.6) {$a$};
\node at (4.33,-1.5) {$b$};
\node at (7,-.6) {$c$};
\node at (3.67,-1.5) {$d$};

\end{scope}

\begin{scope}[shift={(0,-12)}]

\draw (-1,0) -- (8,0);

\foreach \i in {1,2,5,6}
{
\draw[fill] (\i,0) circle (4pt);
}

\draw[fill=black!20] (0,0) circle (4pt);
\draw[fill=black] (-1,0) circle (4pt);
\draw[fill=black!40] (8,0) circle (4pt);
\draw[fill=black!60] (7,0) circle (4pt);

\draw[fill=white] (4,0) circle (6pt);
\draw[fill=white] (3,0) circle (4pt);

\draw[fill=white] (4,0) circle (6pt);
\draw[fill=white] (4,0) circle (4pt);

\node at (0,-.6) {$a$};background
\node at (8,-.5) {$b$};
\node at (7,-.6) {$c$};
\node at (-1,-.5) {$d$};

\end{scope}

%
%
\begin{scope}[shift={(14,0)}]

\draw (1,0) -- (7,0);

\draw (3,0) -- (2.75,-1);
\draw (3,0) -- (3.25,-1);
\draw (4,0) -- (4,-1);
\draw (5,0) -- (5,-1);

\foreach \i in {1,2,5,6,7}
{
\draw[fill] (\i,0) circle (4pt);
}

\draw[fill=black!20] (2.75,-1) circle (4pt);
\draw[fill=black!40] (3.25,-1) circle (4pt);
\draw[fill=black!60] (4,-1) circle (4pt);
\draw[fill=black!80] (5,-1) circle (4pt);

\draw[fill=white] (3,0) circle (4pt);

\draw[fill=white] (4,0) circle (6pt);
\draw[fill=white] (4,0) circle (4pt);

\draw[bend left=20, postaction=decorate] (8, -.5) to (8,-2.5);
\node at (8.75,-1.5) {$S_8$};

\node at (2.75,-1.6) {$a$};
\node at (3.25,-1.5) {$b$};
\node at (4,-1.6) {$c$};
\node at (5,-1.5) {$d$};

\end{scope}

\begin{scope}[shift={(14,-3)}]

\draw (1,0) -- (7,0);

\draw (3,0) -- (2.75,-1);
\draw (3,0) -- (3.25,-1);
\draw (4,0) -- (3.75,-1);
\draw (4,0) -- (4.25,-1);

\foreach \i in {1,2,3,5,6,7}
{
\draw[fill] (\i,0) circle (4pt);
}

\draw[fill=black!20] (2.75,-1) circle (4pt);
\draw[fill=black!40] (3.25,-1) circle (4pt);
\draw[fill=black!60] (3.75,-1) circle (4pt);
\draw[fill=black!80] (4.25,-1) circle (4pt);

\draw[fill=white] (3,0) circle (4pt);

\draw[fill=white] (4,0) circle (6pt);
\draw[fill=white] (4,0) circle (4pt);

\draw[bend left=20, postaction=decorate] (8, -.5) to (8,-2.5);
\node at (8.75,-1.5) {$S_9$};

\node at (2.75,-1.6) {$a$};
\node at (3.25,-1.5) {$b$};
\node at (3.75,-1.6) {$c$};
\node at (4.25,-1.5) {$d$};

\end{scope}

\begin{scope}[shift={(14,-6)}]

\draw (0,0) -- (7,0);

\draw (3,0) -- (3,-1);
\draw (4,0) -- (3.75,-1);
\draw (4,0) -- (4.25,-1);

\foreach \i in {1,2,3,5,6,7}
{
\draw[fill] (\i,0) circle (4pt);
}

\draw[fill=black!20] (0,0) circle (4pt);
\draw[fill=black!40] (3,-1) circle (4pt);
\draw[fill=black!60] (3.75,-1) circle (4pt);
\draw[fill=black!80] (4.25,-1) circle (4pt);

\draw[fill=white] (3,0) circle (6pt);
\draw[fill=white] (3,0) circle (4pt);

\draw[fill=white] (4,0) circle (6pt);
\draw[fill=white] (4,0) circle (4pt);

\draw[bend left=20, postaction=decorate] (8, -.5) to (8,-2.5);
\node at (8.9,-1.5) {$S_{10}$};

\node at (0,-.6) {$a$};
\node at (3,-1.5) {$b$};
\node at (3.75,-1.6) {$c$};
\node at (4.25,-1.5) {$d$};

\end{scope}

\begin{scope}[shift={(14,-9)}]

\draw (-1,0) -- (8,0);

\draw (4,0) -- (4,-1);

\foreach \i in {1,2,3,5,6,7}
{
\draw[fill] (\i,0) circle (4pt);
}

\draw[fill=black!20] (0,0) circle (4pt);
\draw[fill=black!40] (-1,0) circle (4pt);
\draw[fill=black!60] (4,-1) circle (4pt);
\draw[fill=black!80] (8,-0) circle (4pt);

\draw[fill=white] (3,0) circle (6pt);
\draw[fill=white] (3,0) circle (4pt);

\draw[fill=white] (4,0) circle (6pt);
\draw[fill=white] (4,0) circle (4pt);

\node at (0,-.6) {$a$};
\node at (-1,-.5) {$b$};
\node at (4,-1.6) {$c$};
\node at (8,-.5) {$d$};

\end{scope}background

\end{tikzpicture}
\end{center}

\caption{Two different trees that are converted to $P_{10}$ and $Y_{11}$, respectively, using the algorithm and the tree surgeries $S_i$  defined in Section \ref{sec:proof}. The value of $\Tbest$  monotonically increases throughout.  The white vertices are the foci, with $\Tbest$ achieved by the circled vertex. Certain vertices are labeled by $a,b,c,d$ so they can be tracked across steps.}

\label{fig:examples}

\end{figure}
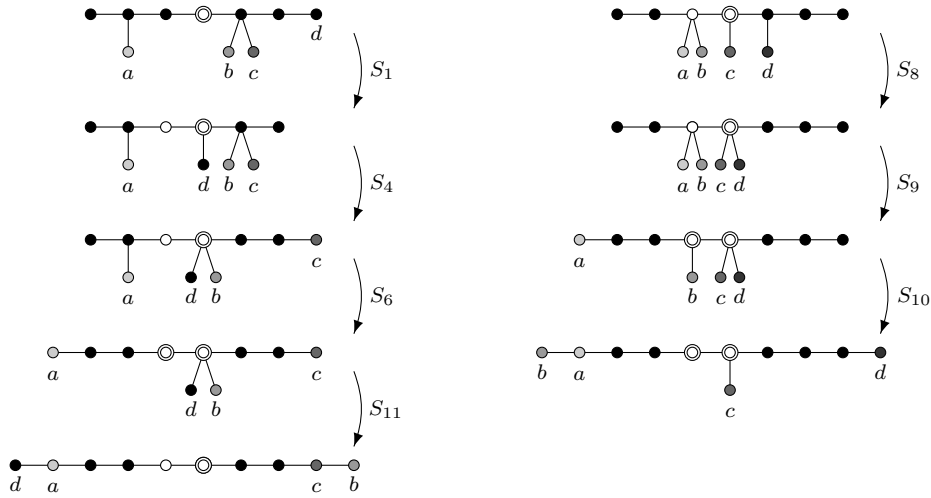

\section{Preliminaries}

For an introduction to the  theory of exact stopping rules, see \cite{lovasz+winkler}. 
To prove Theorem \ref{thm:bestmix}, we need the formula for access times from singleton distributions to the stationary distribution. Therefore, we limit ourselves to describing results from \cite{lovasz+winkler} needed to calculate these values. 

\subsection{Pessimal vertices}
\label{sec:pessimal}

A \emph{$v$-pessimal vertex} $v'$ is a vertex that  satisfies
$H(v',v)=\max_{w\in V}{H(w,v)}.$
Note that pessimal vertices are not necessarily unique; for example, every leaf of the star $S_n$ is pessimal for the central vertex.
We use $\psq{v}$ to denote a pessimal vertex for  $\p{v}$, and
we  employ the notation 
$$\hp(v) = H(\p{v},v) = \max_{w \in V} H(w,v),$$
so that we can refer to the pessimal  hitting time to $v$ in a manner agnostic of the choice of pessimal vertex $\p{v}$.
This allows us to define the lightweight notation
\begin{equation}
\label{eqn:pessdiff}
	\pdelt(v)=\hp_{\nG}(v)-\hp_G(v)
\end{equation}
which shields us from the fact that the $v$-pessimal vertices in $G$ and $\nG$ may be distinct.

\subsection{Calculating mixing times on trees}
Next we address the calculation of  $H(v, \pi)$, the expected length of an optimal stopping rule  from the singleton distribution on $v$ to the stationary distribution $\pi$. We refer to $H(v,\pi)$ as the \emph{mixing time from $v$}. The following result holds for any graph $G$.
\begin{theorem}[\cite{lovasz+winkler}]
\label{thm:mix}
The expected length of an optimal mixing rule starting from vertex $v$ is 
\begin{equation}
\label{eqn:mix}
H(v, \pi) 
= H(\p{v}, v) - \sum_{u \in V} \pi_u H(u,v)
= \hp(v) - \sum_{j \in V} \pi_u H(u,v). 
\end{equation}
\end{theorem}
The very useful formula \eqref{eqn:mix} allows us to calculate the access time $H(v, \pi)$ via a linear combination of vertex-to-vertex hitting times.
In other words, if we know all the pairwise hitting times, it is easy to calculate the access time $H(v, \pi)$ for each  starting vertex $v \in V$. 

For a given graph, we can determine every hitting time $H(u,v)$ by solving a system of linear equations. In the special case of trees,  we have an explicit formula in terms of the distances and degrees of the graph. Variations on this formula appear in the literature (cf. \cite{beveridge, beveridge+wang, dtw, gw} and Chapter 5 of \cite{aldous+fill}). This paper builds on  the results in \cite{beveridge,beveridge+wang}, so we opt for that formulation. From here forward, we assume that $G=(V,E)$ is a tree on $n \geq 4$ vertices.

We start with a well-known result about the hitting time between adjacent vertices in trees.
If $uv \in E$,  then removing this edge breaks $G$ into two disjoint trees $G_u$ and $G_v$ where $u \in V(G_u)$ and $v \in V(G_v)$. We define $V_{u:v} = V(G_u)$ and $V_{v:u} = V(G_v)$. We think of $V_{u:v}$ as the set of vertices that are closer to $u$ than to $v$. For these adjacent vertices, let $F$ denote the induced tree on $V_{u:v} \cup \{   v \}$. We have
\begin{equation} \label{eq:adjhtime}
		H(u,v)= \Ret_{F} (v) - 1 = \sum_{w\in V_{u:v}}{\deg(w)} = 2|V_{u:v}|-1,
\end{equation}
where the first equality holds by the well-known equality $\Ret(v) = 2|E|/\deg(v)$.
Equation \eqref{eq:adjhtime} encodes a very useful property of trees: the hitting time from a vertex $u$ to an adjacent vertex $v$ only depends on $|V_{u:v}|$, independent of the particular structure of the tree.
This equation can be used to  reveal that when  $G=P_n$ is the path ($v_1,v_2,\dots,v_n$), the hitting times are
\begin{equation} \label{eq:htimepath2}
H_{P_n}(v_i,v_j)=
\left\{
\begin{array}{rl}
(i-1)^2-(j-1)^2 & i\leq j,\\
(n-j)^2-(n-i)^2 & i>j.
\end{array}
\right.
\end{equation}
For example, the first of these formulas is calculated by repeatedly using equation \eqref{eq:adjhtime} to evaluate
$H(v_i, v_j) = H(v_i, v_{i+1}) + H(v_{i+1}, v_{i+2}) + \cdots + H(v_{j-1}, v_{j})$. 
A similar argument produces a formula for the hitting times of an arbitrary tree, but first we need some additional notation. Let  $u,v,w \in V$. Define $d(u,v)$ to be the distance between these two vertices and define
\begin{displaymath}
\ell(u,v;w)=\frac{1}{2}(d(u,w)+d(v,w)-d(u,v))
\end{displaymath}
to be the length of the intersection of the $(u,w)$-path and the $(v,w)$-path. 

\begin{lemma}[\cite{beveridge}]
\label{thm:htime}
For any pair of vertices $v_i, v_j$, we have
\begin{equation} \label{eq:htime}
				H(v_i,v_j)=\sum_{v\in V}{\ell(v_i,v;v_j)\deg(v)}.
\end{equation}
\end{lemma}
Theorem \ref{thm:mix} and Lemma \ref{thm:htime} are the fundamental tools in our proof of Theorem \ref{thm:bestmix}. Together they provide a simple way to calculate state-to-state hitting times and mixing times.

\subsection{The foci of a tree}

The center and the barycenter are two well-established notions of centrality.  The center of a tree $G$ is the vertex (or two adjacent vertices) that achieves $\min_{u \in V} \max_{v} d(v,u)$. 
The center does not appear to have any central properties with respect to random walks on trees.
The barycenter is the vertex (or two adjacent vertices) that achieves $\min_{u \in V} \sum_{v \in V} d(v,u)$. 
 Proposition 1 of \cite{beveridge} shows that that the barycenter is the ``average'' center for random walks on trees: this vertex also achieves $\min_{u} \sum_{v \in V} \pi_v H(v,u)$. A third type of centrality for trees was introduced in \cite{beveridge}: the \emph{focus}  is the ``extremal'' center for random walks.

\begin{definition} \label{def:focus} {\bf{(\cite{beveridge})}}
A vertex $v$ achieving $\hp(v) = \min_{w \in V} \hp(w)$ is called a \emph{primary focus} of $G$. If all $v$-pessimal vertices are contained in a single subtree of $G-v$, then the unique $v$-neighbor $u$ in that subtree is also a focus of $G$. If $\hp(u)=\hp(v)$, then $u$ is also a primary focus. Otherwise $u$ is a \emph{secondary focus}. 
\end{definition}

Every tree has either one focus, in which case it is \emph{focal}, or two adjacent foci, in which case it is \emph{bifocal}. 
We have the following two theorems, which relate the foci to mixing walks on $G$.
\begin{theorem}[\cite{beveridge}]
\label{thm:tifocus}
If $G$ is focal with focus $u$, then for all $v$
$
H(v, \pi) = H(v,u) + H(u,\pi).
$
If $G$ is bifocal with foci $u$ and $w$, then for $v \in V_{u:w}$,
$
H(v, \pi) = H(v,u) + H(u,\pi)
$
and for $v \in V_{w:u}$,
$
H(v, \pi) = H(v,w) + H(w,\pi).
$
\end{theorem}
\begin{theorem}[\cite{beveridge}] \label{thm:tbmvertex}
If $H(w,\pi) = \tbm(G)$ then $w$ is a focus of $G$. 
For bifocal $G$ with foci $u$ and $v$, if $H(\p{v},u) < H(\p{u},v)$ then $v$ is the unique vertex achieving $\tbm(G)$.
If $H(\p{v},u) > H(\p{u},v)$ then $u$ is the unique vertex achieving $\tbm(G)$. If $H(\p{v},u) = H(\p{u},v)$ then both vertices achieve $\tbm$.
\end{theorem}

If $H(w,\pi) = \tbm(G)$ then we say that $w$ is a \emph{best mix focus}. Note that a tree can have one or two best mix foci. In the former case, the best mix focus could be either the primary focus or the secondary focus: it depends on the relative sizes of $H(\p{v},u)$ and $H(\p{u},v)$.  This unusual criterion is what makes handling the best mixing time so delicate. A small change to the tree structure can move the location of the best mix focus, even when the foci do not change.

\subsection{The best mixing time for stars, paths and odd wishbones}

\label{sec:bestmix-extreme}

As our first application of these theorems and lemmas, let us calculate the best mixing time for the star, the path (both even and odd $n$) and the wishbone (for odd $n$). This will justify the expressions that appear in Theorem \ref{thm:mix}.
Clearly, the central vertex $v$ of the star $S_n=K_{1,n-1}$ is its unique focus. 
Let $w \neq v$ be any other vertex. We have $\pi_v = 1/2$ and $\pi_w = 1/2(n-1)$. Furthermore, $H(w,v) = \hp(v) = 1$ and obviously $H(v,v)=0$.  Therefore
\[
\tbm (S_n) = \hp(v) - \sum_{w \in V} \pi_w H(w,v) = 1 - \frac{1}{2(n-1)} (n-1) = \frac{1}{2}.
\]
In fact, it is easy to see that $S_n$ is the unique tree on $n$ vertices that achieves this value. Let $G \neq S_n$ and let $v$ be a best mix focus. Then $\pi_v  = \deg(v)/2(n-1)< 1/2$ because the star is the unique tree with a vertex of degree $n-1$.
An optimal mixing rule started at $v$ must exit with probability at least $1 - \pi_v > 1/2$, otherwise the ending distribution will weight vertex $v$ too heavily. This proves that the star is the unique minimizing structure for Theorem \ref{thm:bestmix}.

Next, we calculate $\tbm(P_n)$.  Label the vertices as $(v_1, v_2, \ldots , v_n)$.
We first consider odd $n=2r+1$. By symmetry, the unique focus is $v_{r+1}$. By Theorem \ref{thm:mix} and equation \eqref{eq:htimepath2},
\begin{eqnarray*}
\Tbest(P_{2r+1}) &=& \hp(v_{r+1}) - \sum_{i=1}^{2r+1} \pi_i H(v_i, v_{r+1}) 
\, = \,  \hp(v_{r+1}) - 2 \sum_{i=1}^{r} \pi_i H(v_i, v_{r+1}) \\
&=& r^2 - \frac{1}{2r} r^2 - \frac{1}{r} \sum_{i=2}^r \left( r^2 - (i-1)^2 \right)
\,=\, \frac{2r^2+1}{6} \, = \, \frac{n^2-2n+3}{12},
\end{eqnarray*}
where the second equality follows from the symmetry of the odd path. The calculation for the even path is a bit tougher. Setting $n=2r$, we have two best mix foci $v_r, v_{r+1}$ by symmetry. We will take $v_{r+1}$ as the best mix focus in our calculation. We have
\begin{eqnarray*}
\Tbest(P_{2r}) 
&=& \hp(v_{r+1}) - \sum_{i=1}^{r} \pi_i H(v_i, v_{r+1}) - \sum_{i=r+2}^{2r} \pi_i H(v_i, v_{r+1}) \\
&=& \hp(v_{r+1}) - 2\sum_{i=1}^{r-1} \pi_i H(v_i, v_{r}) - \sum_{i=1}^{r} \pi_i H(v_r, v_{r+1}) \\
&=& r^2 - \frac{2}{4r-2} (r-1)^2 - \frac{4}{4r-2} \sum_{i=2}^{r-1} \left( (r-1)^2 - (i-1)^2 \right) - \frac{1}{2} (2r-1) \\
&=& \frac{2r^2 + 4r-3}{6} \, = \, \frac{n^2 +4n - 6}{12}
\end{eqnarray*}
where we use the symmetry of the path and  $H(v_i, v_{r+1}) = H(v_i, v_r) + H(v_r,v_{r+1})$ for $1 \leq i < r$ in the third equality.

Finally, we determine $\tbm(Y_{2r+1})$. Let the spine be $(w_1, w_2, \ldots, w_{2r})$ with a leaf $x$ adjacent to $w_{r+1}$. Arguing similarly to the even path, we have 
\begin{eqnarray*}
\tbm(Y_{2r+1}) &=& \hp(w_{r+1}) - \sum_{i=1}^{r} \pi_i H(w_i, w_{r+1}) - \sum_{i=r+2}^{2r} H(w_i, w_{r+1}) - \pi_x H(x,w_{r+1}) \\
&=&
\hp(w_{r+1}) - 2\sum_{i=1}^{r} \pi_i H(w_i, w_r) - \sum_{i=1}^{r} H(v_r, w_{r+1}) - \pi_x H(x,w_{r+1}) \\
&=& r^2 - \frac{1}{2r} (r-1)^2 - \frac{1}{r} \sum_{i=2}^{r-1} ((r-1)^2 - (i-1)^2) - \frac{2r-1}{4r} (2r-1) - \frac{1}{4r} \\
&=& \frac{n^2 + 4n - 15}{12}.
\end{eqnarray*}
These calculations show that for odd $n=2r+1 \geq 5$, we have $\tbm(Y_n) = \tbm(P_n)+ (n-3)/2$. Furthermore, we can identify precisely what leads to the wishbone's advantage in equation \eqref{eqn:mix}. The pessimal hitting times to the best mix foci of $Y_{2r+1}$ and $P_{2r+1}$ are both equal to $r^2.$ So the average hitting time to the best mix focus  makes the difference: this quantity is smaller for $Y_n$
 We will see below that this effect can be attributed to the fact that $P_{2n+1}$ has a single focus. This special balance to the tree actually makes it \emph{easier} to mix from the focus.


With these calculations in hand, we have proven the lower bound of Theorem \ref{thm:bestmix}, and calculated the two expressions that appear in the upper bound. The remainder of the paper proves that these upper bounds are correct.

 \section{Caterpillars and Tree Surgeries}
 
 In this section, we show that for a fixed $n \geq 4$, the tree on $n$ vertices that maximizes the best mixing time is a caterpillar.
 Next, we describe our basic techniques, called \emph{tree surgeries}, for making incremental changes to a caterpillars. We introduce some  notation and  useful lemmas for showing that the best mixing time monotonically increases after each tree surgery.

 \subsection{Tree to Caterpillar}

We begin with a lemma showing that $\Tbest$ is maximized by a caterpillar. 
 Recall that a \emph{caterpillar} is a tree such that every vertex is distance at most one from a fixed central path $W = \{ w_1, w_2, \ldots , w_t \}$, called the \emph{spine}. We employ the following notation for talking about sections of the caterpillar. For $2 \leq i \leq t-1$, let $U_i \subset V \backslash W$ denote the set of pendant leaves adjacent to $w_i$. Note that the  spinal leaves are not in $V \backslash W$, so   $w_1 \notin U_2$ and $w_t \notin U_{t-1}$. Finally, we also define  $V_k = U_k \cup \{w_k \}$. See Figure \ref{fig:treetocat} for an example of $w_i$ and $U_i \subset V_i$.

\begin{lemma} \label{lem:cat}
		Let $G$ be a tree on $n$ vertices that is not a caterpillar. There exists a caterpillar $\nG$ on $n$ vertices such that $\tbm(\nG) > \tbm(G)$.
\end{lemma}

\begin{proof}
We give a simple construction of the caterpillar $\nG$; an  example is shown in Figure \ref{fig:treetocat}.

\begin{figure}[t]
	\begin{center}

\begin{tikzpicture}[scale=.6, every node/.style={font=\small}]

\begin{scope}

\draw (0,0) -- (9,0);


\foreach \i in {0, ..., 9}
{
\draw[fill]  (\i,0) circle (4pt);
}

\draw (4,-2.54) -- (3.5,-1.77) -- (3,-2.54);
\draw (3.5,-1.77) -- (3,-1) -- (2.5,-1.77);
\draw (3,-1) -- (3,0);

\draw[fill=black!25]  (3,-1) circle (4pt);
\draw[fill=black!25]  (2.5,-1.77) circle (4pt);
\draw[fill=black!25]  (3.5,-1.77) circle (4pt);
\draw[fill=black!25]  (3,-2.54) circle (4pt);
\draw[fill=black!25]  (4,-2.54) circle (4pt);

\draw (5,-2) -- (5,0);
\draw[fill=black!50]  (5,-1) circle (4pt);
\draw[fill=black!50]  (5,-2) circle (4pt);

\draw (7,-1) -- (7,0);
\draw (7.5,-1.77) -- (7,-1) -- (6.5,-1.77);
\draw[fill=black!75]  (7,-1) circle (4pt);
\draw[fill=black!75]  (6.5,-1.77) circle (4pt);
\draw[fill=black!75]  (7.5,-1.77) circle (4pt);

\node[left] at (5,-1) {$v$}; 
\node[above] at (-.5,0) {$\p{v}$}; 
\node[above] at (9.5,0) {$\psq{v}$}; 

\draw[gray] (7, -1.6) ellipse (.8 and .9);
\draw[gray] (7, -1.1) ellipse (.95 and 1.75);

\draw node at (7, -2.1) {$U_i$};
\draw node at (7, .3) {$w_i$};
\draw node at (8.33,-1.25 ) {$V_i$};

\node at (.5,-1.5) {$G$};

\end{scope}

\node[single arrow, fill=gray, style={font=\tiny}] at (11.5,0) {\phantom{xxx}};

\begin{scope}[shift={(14,0)}]

\draw (0,0) -- (9,0);


\foreach \i in {0, ..., 9}
{
\draw[fill]  (\i,0) circle (4pt);
}

\draw (2.6,-1) -- (3,0) -- (3.4,-1);
\draw (2.2,-1) -- (3,0) -- (3.8,-1);
\draw (3,-1) -- (3,0);

\draw[fill=black!25]  (3,-1) circle (4pt);
\draw[fill=black!25]  (2.6,-1) circle (4pt);
\draw[fill=black!25]  (3.4,-1) circle (4pt);
\draw[fill=black!25]  (2.2,-1) circle (4pt);
\draw[fill=black!25]  (3.8,-1) circle (4pt);

\draw (4.8,-1) -- (5,0) -- (5.2,-1);
\draw[fill=black!50]  (4.8,-1) circle (4pt);
\draw[fill=black!50]  (5.2,-1) circle (4pt);

\draw (7,-1) -- (7,0);
\draw (7.4,-1) -- (7,0) -- (6.6,-1);
\draw[fill=black!75]  (7,-1) circle (4pt);
\draw[fill=black!75]  (7.4,-1) circle (4pt);
\draw[fill=black!75]  (6.6,-1) circle (4pt);

\node[below] at (4.6,-1) {$v$}; 
\node[above] at (-.5,0) {$\p{v}$}; 
\node[above] at (9.5,0) {$\psq{v}$}; 

\draw[gray] (7, -1.25) ellipse (.75 and .7);
\draw[gray] (7, -.85) ellipse (.9 and 1.5);

\draw node at (7, -1.6) {$U_i$};
\draw node at (7, .3) {$w_i$};
\draw node at (8.25,-1 ) {$V_i$};

\node at (.5,-1.5) {$\nG$};

\end{scope}

\end{tikzpicture}

		\caption{Transforming a tree $G$ into a caterpillar $\nG$ with $\Tbest(G) < \Tbest(\nG)$. Every vertex in $U_i$ becomes adjacent to $w_i$.}
		\label{fig:treetocat}
	\end{center}
\end{figure}
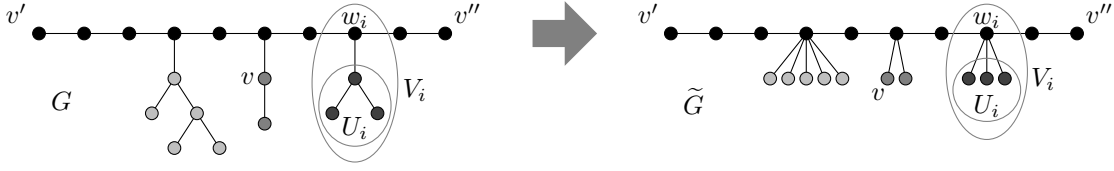

Choose any $v\in V(G)$. Recall from Section \ref{sec:pessimal} that $\p{v}$ is a $v$-pessimal vertex, and that $\psq{v}$ is a $\p{v}$-pessimal vertex. Let 
$
	W=\{\p{v}=w_1,w_2,\dots,w_t=\psq{v}\}
$
be the vertices on the unique $(\p{v},\psq{v})$ path $P \subset G$. These vertices will be the spine of the caterpillar $\nG$. (Note that $v$ is not necessarily a member of $W$.) We replace the subforest attached to each spine vertex $w_i$ with a set of leaves adjacent to $w_i$. 
Define $\nG$ to be the caterpillar with vertex set $V(G)$ and edge set
$
E(\nG)=  E(P) \cup \{uw_i\mid u\in U_i, 1<i<t\}.
$

The foci of $\nG$ are contained in $W$ because leaves cannot be   foci  when $n \geq 3$.  Without loss of generality, let $w_r$ be the best mix focus of $\nG$ and let $\p{w_r} = w_1$. If $G$ is bifocal, the other focus must be $w_{r-1}$ and $\p{w_{r-1}} = w_t$. For $1\leq i\leq t-1$, the quantity   
$$\sum_{v \in V_{w_i:w_{i+1}}} \deg_G(v) = 2 |V_{w_i:w_{i+1}}| - 1 = \sum_{v \in V_{w_i:w_{i+1}}} \deg_{\nG}(v) $$ 
does not change. 
By equation \eqref{eq:adjhtime} we have
\begin{equation}
\label{eqn:nochange}
H_G(w_i, w_j) = H_{\nG}(w_i,w_j)
\end{equation}
for all $1 \leq i,j \leq t$. Considering any non-spine vertex $v \in U_i$, we have
$$
H_{\nG}(v, w_r) - H_G(v,w_r) 
= 
 1 + H_{\nG}(w_i, w_r) -   \big( H_{G}(v, w_i) + H_{G}(w_i, w_r)  \big)
= 1 - H_G(v,w_i)  \leq 0.  
$$
We conclude that $w_1$ is $v_r$-pessimal for both $G$ and $\nG$. Likewise, $w_t$ is $w_{r-1}$-pessimal for both $G$ and $\nG$.  
By equation \eqref{eqn:nochange} and Theorem \ref{thm:tbmvertex}, the vertex $w_r$ is the best mix focus of $\nG$ and $\hp_G(w_r) = \hp_{\nG}(w_r)$.  Finally, 
Theorem \ref{thm:mix} yields
\begin{eqnarray*}
\lefteqn{ \Tbest(\nG) - \Tbest(G) \, = \,
H_{\nG}(w_r, \npi) - H_{G}(w_r, \pi)  
}
\\
&=& \sum_{i=1}^{t} \sum_{v \in V_i} \left(  \pi_{v}  (H_G (v, w_i)+ H_G (w_i, w_r)) - \npi_{v} (H_{\nG} (v, w_i) + H_{\nG} (w_i, w_r)) \right) \\
&=& \sum_{i=1}^{t} \sum_{v \in V_i}  \left(  \pi_{v}  H_G (v, w_i) - \npi_{v} H_{\nG} (v, w_i)  \right)
\, = \,
\sum_{i=2}^{t-1} \sum_{v \in U_i}  \left(  \pi_{v}  H_G (v, w_i) - \npi_{v} H_{\nG} (v, w_i)  \right)
\,  \geq \, 0
\end{eqnarray*}
because each $ \sum_{v \in V_i}   \pi_{v}  H_G (v, w_i) - \npi_{v} H_{\nG} (v, w_i)  \geq 0$, with equality holding if and only if every $v \in U_i$ is already adjacent to $w_i$. Since $G$ was not a caterpillar, there must be at least one exceptional $v$ in some $U_i$, so  this inequality is strict.
\end{proof}
		

\subsection{Leaf Transplants and Tree Surgeries}

For the remainder of the paper, $G$ is a caterpillar on $n$ vertices with spine  $W= \{ w_1, w_2, \ldots , w_t \}$. 
We choose to view $w_1$ as the leftmost vertex and $w_t$ as the rightmost vertex.
For $v \in V$, we define the \emph{left pessimal hitting time} $\L(v) = H(w_1,v)$ and the \emph{right pessimal hitting time} $\R(v) = H(w_t,v)$. Of course, $\hp(v) = \max \{ \L(v), \R(v) \}$. When $G$ is bifocal, we always label the foci as $w_{r-1}$ and $w_r$ where $w_r$ achieves $H(w_r,\pi)=\Tbest(G)$. In this case, we view $\{ w_1, \ldots , w_{r-2} \}$ as the \emph{left spine} and $\{ w_{r+1}, \ldots , w_{t} \}$ as the \emph{right spine}.  The foci $\{ w_{r-1}, w_r \}$ are considered the \emph{central spine}. Note that this arrangement guarantees that $\p{w_r}=w_1$ and $\p{w_{r-1}}=w_t$. We collect some basic  results about  caterpillar foci in the next lemma.

\begin{lemma}
\label{lemma:focuscat}
Let $G$ be a caterpillar with spine $W= \{ w_1, w_2, \ldots , w_t \}$. Then 
\begin{enumerate}
\item[(a)] The vertex $w_r$ is the unique focus of $G$ if and only if $\L(w_r)=\R(w_r)$.
\item[(b)] The foci of $G$ are $w_{r-1}, w_r$ if and only if  $\L(w_r)>\R(w_r)$ and  $\R(w_{r-1})>\L(w_{r-1}).$
\item[(c)] If $G$ has foci $w_{r-1}, w_r$ then vertex $w_r$   is a best mix focus if and only if $\L(w_{r-1})\leq \R(w_{r}).$
\item[(d)] 
If $L(w_r) > R(w_r)$ and $\L(w_{r-1})\leq \R(w_{r})$ then $w_r$ is a best mix focus and $w_{r-1}$ is also a focus of $G$.
\end{enumerate}
\end{lemma}

\begin{proof}
Parts (a) and (b) are a direct consequence of Definition \ref{def:focus}. Part (c) is restatement of Theorem \ref{thm:tbmvertex} for caterpillars.  For part (d), we observe that 
$\L(w_{r-1})\leq \R(w_r) < \R(w_{r-1})$, so that $w_{r-1}$ and $w_r$ satisfy (b) and (c).
\end{proof}

When $U_i \neq 0$, we define the \emph{leaf transplant} $\tr{i}{j}$ as the relocation of  $x \in U_i$ so that this leaf is now adjacent to $w_j$ where $1 \leq j \leq t$. Usually, we have $2 \leq j \leq t-1$, so that the leaf $x$ becomes an element of $U_j$. Our most common operation will be an \emph{elementary leaf transplant} where $j=i \pm 1$. Occasionally, we will take $j \in \{1, t\}$, so that the leaf transplant actually increases the length of the spine.
Finally, we need one special transplant that reduces the length of the spine by relocating the leaf $w_t$ to become adjacent to the best mix focus $w_r$: we use $\str{t}{r}$ to denote this rare spinal leaf transplant. The more general term \emph{tree surgery} denotes either a single leaf transplant (either $\tr{i}{j}$ or $\str{t}{r}$), or a pair of leaf transplants $\tr{i_1}{j_1} \wedge \tr{i_2}{j_2}$ performed simultaneously. We use $\S(G)$ to denote the caterpillar that results from applying tree surgery $\S$ to caterpillar $G$. Our algorithm uses eleven different tree surgeries. They are enumerated in Tables \ref{table:phase1}, \ref{table:phase2} and \ref{table:phase3} below. 

The remainder of this section is devoted to proving some fundamental results about the effects of tree surgeries on the best mixing time. 
Let $G=(V,E)$ be a caterpillar and let  $\nG=\S(G) = (V, \nE)$ be the caterpillar obtained by  applying  tree surgery $\S$, with
stationary distribution $\npi = \pi_{\nG}$.  We introduce notation for the comparison of  various caterpillar measurements.
Just as with the pessimal notation introduced in Section \ref{sec:pessimal},  we are motivated  to shield ourselves from the particular locations of important vertices (like foci and pessimal vertices) which may be different in $G$ and $\nG$.

As with  equation \eqref{eqn:pessdiff}, we introduce $\Delta$-notation to compactly represent the changes in various quantities due to a tree surgery. For vertices $v, v_1, v_2 \in V$, we define:
$$
\begin{array}{ccc}
\begin{array}{rcl}
\Delta \deg(v) &=&  \deg_{\nG}(v)  - \deg_G(v), \\
 \Delta \pi(v) &=&   \npi(v) - \pi(v) , \\
\Delta H(v_1, v_2) &=&  H_{\nG}(v_1, v_2) - H_G(v_1, v_2) ,  \\
\Delta H(\pi, v) &=& H_{\nG}(\npi, v) - H_G(\pi, v),  \\
\end{array}
&\quad&
\begin{array}{rcl}
\Delta \hp(v) &=& \hp_{\nG}(v) - \hp_G(v), \\
\Delta \L(v) &=& \L_{\nG}(v) - \L_G(v), \\
\Delta \R(v) &=& \R_{\nG}(v) - \R_G(v), \\
\Delta \Tbest &=& \Tbest(\nG) - \Tbest(G). \\
\end{array}
\end{array}
$$

Let $G$ be a caterpillar and let $\nG=\S(G)$ be the result of a leaf transplant $\tr{i}{j}$. We give formulas for  $\Delta H(v,w_r)$ for use  in later sections. We only consider  leaf transplants $\tr{i}{j}$ where either $1 \leq i,j \leq r$ or where $r+1 \leq i,j \leq t$ (so that we never transplant leaves from the left spine to the right spine, or vice versa).
The formulas below cover two qualitatively different cases for $\tr{i}{j}$. When $2 \leq j \leq t-1$, the spine length is unaffected.  However,  when $j \in \{ 1, t \}$,  the spine length increases. The formulas below still hold, though the arguments are slightly different.
All four formulas below follow quickly from equation \eqref{eq:adjhtime}, and we leave these short calculations to the reader. Figure \ref{fig:transplants} shows the effect of two example transplants on hitting times. The arguments for $j=i \pm 1$ are straight forward, and the other cases follow inductively.

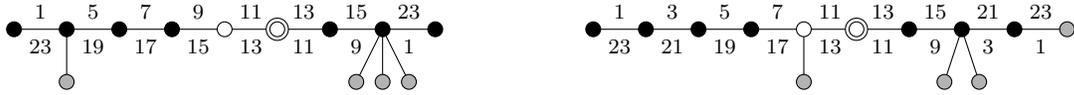
\begin{figure}[t]

\begin{center}
\begin{tikzpicture}[scale=.7, every node/.style={font=\scriptsize}]

\begin{scope}

\draw (1,0) -- (9,0);

\draw (2,0) -- (2,-1);

\draw (8,0) -- (8,-1);
\draw (8,0) -- (8.5,-1);
\draw (8,0) -- (7.5,-1);

\foreach \i in {1,2,3,4,7,8,9}
{
\draw[fill] (\i,0) circle (4pt);
}

\draw[fill=black!30] (2,-1) circle (4pt);
\draw[fill=black!30] (7.5,-1) circle (4pt);
\draw[fill=black!30] (8,-1) circle (4pt);
\draw[fill=black!30] (8.5,-1) circle (4pt);

\draw[fill=white] (5,0) circle (4pt);

\draw[fill=white] (6,0) circle (6pt);
\draw[fill=white] (6,0) circle (4pt);

\node[above] at (1.5,0) {1};
\node[above] at (2.5,0) {5};
\node[above] at (3.5,0) {7};
\node[above] at (4.5,0) {9};
\node[above] at (5.5,0) {11};
\node[above] at (6.5,0) {13};
\node[above] at (7.5,0) {15};
\node[above] at (8.5,0) {23};

\node[below] at (1.5,0) {23};
\node[below] at (2.5,0) {19};
\node[below] at (3.5,0) {17};
\node[below] at (4.5,0) {15};
\node[below] at (5.5,0) {13};
\node[below] at (6.5,0) {11};
\node[below] at (7.5,0) {9};
\node[below] at (8.5,0) {1};

\end{scope}

\begin{scope}[shift={(11,0)}]

\draw (1,0) -- (10,0);

\draw (5,0) -- (5,-1);

\draw (8,0) -- (8.33,-1);
\draw (8,0) -- (7.67,-1);

\foreach \i in {1,2,3,4,7,8,9}
{
\draw[fill] (\i,0) circle (4pt);
}

\draw[fill=black!30] (5,-1) circle (4pt);
\draw[fill=black!30] (7.67,-1) circle (4pt);
\draw[fill=black!30] (8.33,-1) circle (4pt);
\draw[fill=black!30] (10,0) circle (4pt);

\draw[fill=white] (5,0) circle (4pt);

\draw[fill=white] (6,0) circle (6pt);
\draw[fill=white] (6,0) circle (4pt);

\node[above] at (1.5,0) {1};
\node[above] at (2.5,0) {3};
\node[above] at (3.5,0) {5};
\node[above] at (4.5,0) {7};
\node[above] at (5.5,0) {11};
\node[above] at (6.5,0) {13};
\node[above] at (7.5,0) {15};
\node[above] at (8.5,0) {21};
\node[above] at (9.5,0) {23};

\node[below] at (1.5,0) {23};
\node[below] at (2.5,0) {21};
\node[below] at (3.5,0) {19};
\node[below] at (4.5,0) {17};
\node[below] at (5.5,0) {13};
\node[below] at (6.5,0) {11};
\node[below] at (7.5,0) {9};
\node[below] at (8.5,0) {3};
\node[below] at (9.5,0) {1};

\end{scope}

\end{tikzpicture}

\end{center}

\caption{The effect of transplants on hitting times. The right tree results from applying $\tr{2}{5} \wedge \tr{8}{9}$ on the left tree. Each spine edge $(w_i, w_{i+1})$ is labeled by $H(w_i,w_{i+1})$ above and $H(w_{i+1},w_i)$ below. Other hitting times are linear combinations of those shown.}

\label{fig:transplants}

\end{figure}

Suppose that the leaf transplant $\tr{i}{j}$ moves $x$ from  $U_i$ to $U_j$. For $v \in V_k$,  the value of $\Delta H (v,w_r)$ depends on the relative locations of $i,j$ and $k$:
\begin{align}
\label{eq:leftleft2}
\mbox{If } 1\leq j<i\leq r \mbox{ then } &
		\Delta H(v,w_r)=
		\left\{
			\begin{array}{cl}
				H_G(w_j,w_i)+2(i-j) & \quad v=x,\\
				2(i-j) & \quad k\leq j,\\
				2(i-k) & \quad j<k<i,\\
				0 & \quad k \leq i \leq r \mbox{ and } v \neq x.
			\end{array}
		\right.
\\
\label{eq:leftright2}
\mbox{If }   1\leq i<j\leq r  \mbox{ then } &
		\Delta H(v,w_r)=
		\left\{
			\begin{array}{cl}
				-H_G(w_i,w_j) & \quad v=x,\\
				-2(j-i) & \quad k\leq i \mbox{ and } v \neq x,\\
				-2(j-k) & \quad i<k<j,\\
				0 & \quad j \leq k \leq r.
			\end{array}		
		\right.
\\
\label{eq:rightleft2}
\mbox{If } r \leq j<i\leq t  \mbox{ then } &
		\Delta H(v,w_r)=
		\left\{
			\begin{array}{cl}
				-H_G(w_i,w_j) & \quad v=x,\\
				-2(i-j) & \quad i \leq k \mbox{ and } v \neq x,\\
				-2(i-k) & \quad j < k < i,\\
				0 & \quad r < k \leq j.
			\end{array}
		\right.
\\
 \label{eq:rightright2}
 \mbox{If } r \leq  i<j\leq t  \mbox{ then } &
 		\Delta H(v,w_r)=
		\left\{
			\begin{array}{cl}
				H_G(w_j,w_i)+2 & \quad v=x,\\
				2(j-i) & \quad j \leq k,\\
				2(j-k) & \quad i < k < j,\\
				0 & \quad r < k \leq i \mbox{ and } v \neq x.
			\end{array}
		\right.
\end{align}
We get similar equations for hitting times to the second focus $w_{r-1}$, with the appropriate change to the bounds on the indices $i,j$.  This completes our list of useful hitting time changes due to common leaf transplants.

Next, we turn our attention to tracking the change in the best mixing time after a tree surgery. The lemmas that follow will be used to analyze all eleven surgeries used in our algorithm. 
If $w_r$ is the best mix focus of $G$ and $w_s$ is the best mix focus of $\nG$, then 
\begin{equation}
\label{eq:tbmchange}
\Delta \tbm = \hp_{\nG}(w_s) - \hp_G(w_r) - \big(  H_{\nG}(\npi, w_s) - H_G(\pi, w_r)   \big).
\end{equation}
As we alter our caterpillar, we strive to maintain $w_r$ as the best mix focus and the leftmost spinal vertex as the $w_r$-pessimal vertex. In this case,  equation \eqref{eq:tbmchange} simplifies to
$
\Delta \tbm  = \Delta \hp(w_r) - \Delta H(\pi, w_r) =  \Delta L(w_r) - \Delta H(\pi, w_r).
$
This can be rephrased as the simple criterion:
\begin{equation}
	 \label{eq:tbmchange2}
		\Delta \tbm  \geq 0 \quad \Longleftrightarrow \quad \pdelt(w_r) \geq \Delta H(\pi,w_r).
\end{equation}

The next two results track the effect of a surgery on the location of the foci.
	\begin{lemma} 
	\label{lem:criteria}
		Let $G$ be a caterpillar. Let $\S$ be a tree surgery and let $\nG=\S(G)$. If 
\begin{eqnarray}
 \label{eq:lchange}
L_G(w_r)-R_G(w_r) &>& \Delta\R(w_r)-\Delta\L(w_r) \\
\mbox{and} \quad
\label{eq:newtbmfoc}
R_G(w_r)-L_G(w_{r-1}) &\geq& \Delta\L(w_{r-1})-\Delta\R(w_r),
\end{eqnarray}
%
		then $\nG=\S(G)$  has best mix focus $w_{r}$ and $w_{r-1}$ is also a focus. 
	\end{lemma}
	
\begin{proof}
These criteria are equivalent to the conditions of  Lemma \ref{lemma:focuscat} (d) for $\nG$.  
\end{proof}	
	
\begin{corollary} \label{cor:criteria}
Let $G$ be a bifocal caterpillar with focus $w_{r-1}$ and best mix focus $w_r$. Let $\S$ be a tree surgery and let $\nG=\S(G)$. If 
$\Delta\R(w_r)-\Delta\L(w_r) \leq 0$
and
$\Delta\L(w_{r-1})-\Delta\R(w_r)\leq 0$
%
then $\nG=\S(G)$ also has focus $w_{r-1}$ and best mix focus $w_r$. 
\end{corollary}

\begin{proof}
This follows directly from Lemma \ref{lem:criteria} and the inequalities   $L(w_r)-R(w_r)\geq 1$ and $R(w_r)-L(w_{r-1})\geq 0$.
\end{proof}

Lemma \ref{lem:criteria} and Corollary \ref{cor:criteria} are key tools of our methodology. We use them repeatedly to verify the foci and best mix focus of our caterpillar after a surgery has been applied. Once we know that $w_r$ is still the best mix focus, we then verify that equation \eqref{eq:tbmchange2} holds. 
Next, we give a final test for $\Delta \Tbest \geq 0$. This lemma will be used very frequently in subsequent sections.

\begin{lemma} \label{lem:slack}
Let $G$ be a caterpillar and let $\S$ be a tree surgery such that $w_r$ is a best mix focus of $G$ and $\nG$.
Let $A=\{v\in V\mid \Delta H(v,w_r)>\pdelt(w_r)\}$ and let $C = \{ v \in \overline{A} \mid \Delta(v) > 0 \}$.  If there exists $C \subset B \subset \overline{A}$  for which 
			\begin{align}
		\label{eqn:slack}	
				&\sum_{v\in A}{ \Big(\deg_{\nG}(v)(\Delta H(v,w_r)-\pdelt(w_r))+\Delta\deg_G(v)H_G(v,w_r)  \Big)} \nonumber \\
					\leq&\sum_{u\in B}{ \Big( \deg_{\nG}(u)(\pdelt(w_r)-\Delta H(u,w_r))-\Delta\deg_G(u)H_G(u,w_r) \Big)}
			\end{align}
		then $\Delta \tbm \geq 0$.
\end{lemma}

Before proving this lemma, we make a few comments its use in later sections.
First, the set $A$ will always be small: it will only contain one or both of the leaves moved during the surgery.
Second, if the lemma holds for $B$, then it holds for any superset of $B$, including $\overline{A}$. However, there is no need to calculate the right hand side for $\overline{A}$ when a small subset will do.
For our earlier surgeries, the set $B$ will be small, typically consisting of a handful of spine vertices near the transplant locations. In later surgeries, $B$ will consist of half or  all of the spine. 

\begin{proof} We  decompose $\Delta H(\pi,w_r)$ as follows:
	\begin{align}
		\Delta H(\pi,w_r)	&=\sum_{v\in V}{\big(\npi(v)H_{\nG}(v,w_r)-\pi(v)H_G(v,w_r)\big)}\nonumber \\
		&=\sum_{v\in V}\big(\npi(v)(\Delta H(v,w_r) + H_G(v,w_r))-\pi(v)H_G(v,w_r)\big)\nonumber \\
		&=\sum_{v\in V}{\big(\npi(v)\Delta H(v,w_r)+\Delta\pi(v)H_G(v,w_r)\big)}. \nonumber
	\end{align}
Let $g(v) =  \deg_{\nG}(v)(\pdelt(w_r) -  \Delta H(v,w_r)) - \Delta\deg_G(v)H_G(v,w_r)$. We can rewrite
	\begin{align*} 
 \pdelt(w_r)-\Delta H(\pi,w_r)
		& =  \pdelt(w_r)-\sum_{v\in V}{\left(\npi(v)\Delta H(v,w_r)+\Delta\pi(v)H_G(v,w_r)\right)}\\
		&= \frac{1}{2|E|} \left( \sum_{v\in A} g(v) + \sum_{u\in \overline{A}} g(u) \right) \,
		 \geq \,\frac{1}{2|E|} \left(   \sum_{v\in A} g(v) + \sum_{u\in B} g(u) \right). 
	\end{align*}
The final inequality holds because $g(u) \geq 0$ for every $u \in \overline{A} \backslash B \subset \overline{A} \backslash C$. 
Therefore, if $ - \sum_{v\in A} g(v) \leq  \sum_{u\in B} g(u)$ then $ \pdelt(w_r)-\Delta H(\pi,w_r) \geq 0$.
\end{proof}

We have an immediate corollary  for the special case where $A$ is empty.
	\begin{corollary} \label{cor:slack}
		Given a tree $G$ and any tree surgery $\S$ such that $w_r$ is the best mix focus of $G$ and $\nG$, if 
$
				\Delta H(v,w_r)\leq\pdelt(w_r),
$
		for all $v\in V$, then $\Delta \tbm \geq 0$.
		\qed
	\end{corollary}

Finally, we note that  if the $w_r$-pessimal vertices in $G$ and $\nG$ are the leftmost spinal vertices in these graphs,  then we can replace $\pdelt(w_r)$  with $\Delta\L(w_r)$ in Lemma \ref{lem:slack} and Corollary \ref{cor:slack}.

\section{Proof of the upper bound in Theorem \ref{thm:bestmix}}
 \label{sec:proof}

In Section \ref{sec:bestmix-extreme}, we proved the lower bound of Theorem  \ref{thm:bestmix}, and calculated the best mix time for the even path and the odd wishbone. In this section, we prove  the upper bound of Theorem \ref{thm:bestmix}, leaving the details to the lemmas in the accompanying subsections.
Two examples of the algorithm in practice are shown in Figure \ref{fig:examples} above.

\begin{proofof}{Theorem \ref{thm:bestmix}}
Given a tree $G$ that is not an even path or an odd wishbone, we must show that it does not maximize the best mixing time. We may assume that $G$ is not a star since that is clearly not the maximizing structure.
Starting from our tree $G$, we let $G_0$ be the caterpillar constructed by Lemma \ref{lem:cat}, so that $\Tbest(G) \leq \Tbest(G_0)$. Next, we apply a sequence of tree surgeries to produce  a sequence of caterpillars $G_0,G_1, G_2, \ldots, G_m$ such that $\Tbest(G_{i-1}) \leq \Tbest(G_i)$,  where the final caterpillar $G_m$ is either $P_n$ or $Y_n$. This occurs in an algorithmic manner, divided into three phases.

\begin{figure}[t]

\begin{center}
\begin{tikzpicture}[scale=.66]

\node at (0,-2.5) {(c)};

\draw (-1.5,1) -- (0,0)  -- (2.25, 1.5);
\draw (0,0) -- (0,-1);

\draw[fill] (-1.5,1) circle (3pt);
\draw[fill] (-.75,.5) circle (3pt);
\draw[fill] (0,0) circle (3pt);
\draw[fill] (.75,.5) circle (3pt);
\draw[fill] (1.5,1) circle (3pt);
\draw[fill] (2.25,1.5) circle (3pt);
\draw[fill] (0,-1) circle (3pt);

\begin{scope}[shift={(-5,-.5)}]

\node at (.75,-2) {(b)};

\draw (0,3) -- (0,0) -- (1.5,0) -- (1.5,2);

\draw (0,-1) -- (0,0);
\draw (.33,-1) -- (0,0);
\draw (-.33,-1) -- (0,0);

\draw (1.5,-1) -- (1.5,0);
\draw (1.83,-1) -- (1.5,0);
\draw (2.16,-1) -- (1.5,0);
\draw (1.17,-1) -- (1.5,0);

\draw[fill] (0,3) circle (3pt);
\draw[fill] (0,2) circle (3pt);
\draw[fill] (0,1) circle (3pt);
\draw[fill] (0,0) circle (3pt);
\draw[fill] (1.5,0) circle (3pt);
\draw[fill] (1.5,1) circle (3pt);
\draw[fill] (1.5,2) circle (3pt);
.3

\draw[fill]  (0,-1) circle (3pt);
\draw[fill]  (.33,-1) circle (3pt);
\draw[fill]  (-.33,-1) circle (3pt);

\draw[fill]  (1.5,-1) circle (3pt);
\draw[fill]  (1.83,-1) circle (3pt);
\draw[fill]  (2.16,-1) circle (3pt);
\draw[fill]  (1.17,-1) circle (3pt);

\end{scope}

\begin{scope}[shift={(-12,0)}]

\node at (.5,-2.5) {(a)};

\draw (-3,0) -- (0,0) -- (1,0) -- (5,0);

\draw (0,-1) -- (0,0);
\draw (.33,-1) -- (0,0);
\draw (-.33,-1) -- (0,0);

\draw (1.2,-1) -- (1,0);
\draw (.8,-1) -- (1,0);

\draw (3,1) -- (3,0);
\draw (-2,1) -- (-2,0);

\draw[fill] (-3,0) circle (3pt);
\draw[fill] (-2,0) circle (3pt);
\draw[fill] (-1,0) circle (3pt);
\draw[fill] (0,0) circle (3pt);
\draw[fill] (1,0) circle (3pt);
\draw[fill] (2,0) circle (3pt);
\draw[fill] (3,0) circle (3pt);
\draw[fill] (4,0) circle (3pt);
\draw[fill] (5,0) circle (3pt);

\draw[fill] (-2,1) circle (3pt);
\draw[fill] (3,1) circle (3pt);

\draw[fill]  (0,-1) circle (3pt);
\draw[fill]  (.33,-1) circle (3pt);
\draw[fill]  (-.33,-1) circle (3pt);

\draw[fill]  (1.2,-1) circle (3pt);
\draw[fill]  (.8,-1) circle (3pt);

\end{scope}

\end{tikzpicture}
\end{center}

\caption{Examples of the  milestone graphs obtained at the end of each phase of our process: (a) a seesaw,  (b) a twin broom, and (c) the  wishbone $Y_7$.}
\label{fig:menagerie}

\end{figure}
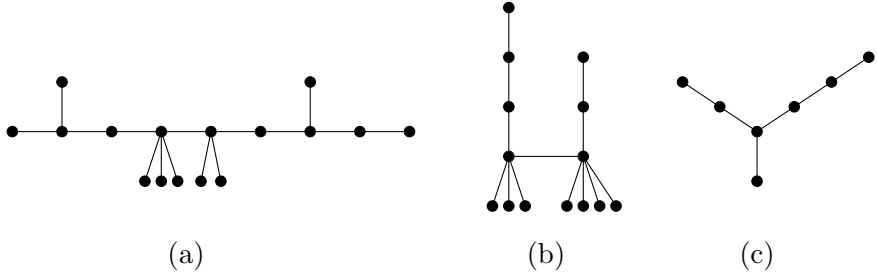

In order to define these phases, we define two special subfamilies of caterpillars which  mark milestones in our graph sequence. Examples of these graphs and a wishbone are shown in Figure \ref{fig:menagerie}.
A \emph{twin broom} is a bifocal caterpillar whose only non-spine leaves are adjacent to the foci $w_{r-1}, w_r$. In other words, $V= W \cup U_{r-1} \cup U_r$. A \emph{seesaw} is a twin broom that has at most one additional leaf on each side of the spine.  In other words, a seesaw satisfies $| \cup_{i=2}^{r-2} U_i| \leq 1$ and $| \cup_{i=r+1}^{t-1} U_i| \leq 1$.
Figure \ref{fig:examples} includes examples of each. In that figure, the output of $S_4$ and the initial tree on 11 vertices are both seesaws. Meanwhile, the outputs of $S_6$ and $S_8$ are both twin brooms (as are all graphs that follow them).

We start with a caterpillar $G$ on $n$ vertices. In Phase One, we convert the caterpillar into a seesaw. Phase Two converts the seesaw into a twin broom. Phase Three converts the twin broom into one of  $P_{n}$ (when $n$ is even) or $Y_n$ (when $n$ is odd).  Lemmas \ref{lemma:cat-to-seesaw}, \ref{lemma:seesaw-to-twinbroom}, and \ref{lemma:twinbroom-to-end}  below show that the best mixing time monotonically increases in every phase of this process. This proves Theorem \ref{thm:bestmix}.
\end{proofof}

Our caterpillar transformation consists of incremental steps that move one or two leaves at a time. This  allows us to monitor the delicate balance maintained by the best mix focus.  
In particular,  there may be a critical step at which we change the focus that attains the best mixing time. This happens in one of two ways. Usually, we make a small change that keeps the current best mix focus, but also causes a neighbor  to also become a best mix focus. The other change is more abrupt: Surgery $S_5$ below changes the location of the unique best mix focus to a neighbor of the current one. In $S_5$, the prescribed caterpillar structure during that surgery makes this crucial transition manageable. 

The remainder of this paper is devoted to proving Lemmas \ref{lemma:cat-to-seesaw}, \ref{lemma:seesaw-to-twinbroom}, and \ref{lemma:twinbroom-to-end}.

\subsection{Phase One: Caterpillar to Seesaw}

In this subsection, we prove that Phase One is successful: we can transform any caterpillar into a seesaw while also increasing the best mixing time. Let $G$ be a caterpillar with spine $W=\{w_1, w_2, \ldots, w_t \}$ where $w_r$ is the best mix focus and $\p{w_r}=w_1$.

\begin{lemma} \label{lemma:cat-to-seesaw}
Let $G$ be a caterpillar $G$ on $n$ vertices. Phase One creates  a seesaw $\nG$ such that $\tbm(\nG)\geq\tbm(G)$.
\end{lemma}

\begin{proof}
If $G$ is already a seesaw then $\nG=G$. 
Table \ref{table:phase1} shows the five surgery types employed during Phase One.  We defer the proofs that $\Delta \Tbest \geq 0$ for each of these surgeries to the lemmas that follow.

Figure \ref{fig:phase1} shows the workflow for Phase One. First, if the caterpillar has a unique focus then we use $\S_1$ to create a caterpillar with two foci. From here forward, the caterpillar will remain bifocal.  Let $w_{r-1}, w_r$ be the foci of a bifocal caterpillar with $w_r$ achieving $\Tbest$. A leaf $x \in V \backslash W$ is \emph{good} when $x \in U_{2} \cup U_{r-1} \cup U_{r} \cup U_{t-1}$. All other leaves in $V \backslash W$ are \emph{bad}. First, we repeatedly use $\S_2$ to move pairs of bad leaves on the same side of the spine (one towards the end and the other towards the center). This loop terminates whether there is at most one bad leaf on each side of the spine. At this point, we use $S_3$ to  extend the left spine and to transplant a leaf to $U_r$. This requires  there are at least two left leaves (which can only happen when  at least one is in $U_2$ since we are done with $\S_2$). After extending the spine, the previously good leaves at $U_2$ become bad. This throws us back into the $\S_2$ loop.

We exit the  $S_3$ loop when there is at most one left leaf.  At this point, we deal with the right leaves. When $L(w_r) > R(w_r) + 1$, we apply $S_4$, the right-hand surgery analogous to $S_3$. However, if $L(w_r)= R(w_r)+1$ then applying $S_4$ would create a focal caterpillar, which we choose to avoid. Instead, we apply $S_5$, which transplants a single right leaf to the end of the right spine. Applying either $S_4$ or $S_5$ might create bad right leaves, which puts us back into the $S_2$ loop. Surgeries $S_3, S_4, S_5$ reduce the number of non-spinal vertices, so Phase One must terminate. 
Ultimately, we create a caterpillar with at most one left leaf and at most one right leaf, while $U_{r-1} \cup U_{r}$ may contain many leaves. This is a seesaw graph, as desired.
\end{proof}


\begin{table}[ht]
\begin{center}
\begin{tabular}{|c|p{2.5in}|@{}c@{}|}

\hline

Surgery &  Initial Conditions & Illustration\\
\hline
$\S_1$ & \small $G$ is focal, so $L(w_r) = R(w_r)$. The transplant depends on whether $U_{t-1} = \emptyset$. &

\begin{tabular}{c}
\begin{tikzpicture}[scale=0.5,
    decoration={
      markings,
      mark=at position 1 with {\arrow[scale=1.25,black]{latex}};
    }
  ]

\begin{scope}[shift={(0,-2)}]

\draw[gray!70] (2,0) -- (2,-1);
\draw[gray!70, fill=gray!70] (2,-1) circle (4pt);

\draw[dashed] (0,0) -- (5,0);

\draw (0,0) -- (.5,0);
\draw (1.5,0) -- (2.5,0);
\draw (3.5,0) -- (5,0);

\foreach \i in  {0,4,5} 
{
\draw[fill] (\i,0) circle (4pt);
}

\draw[fill=white] (2,0) circle (6pt);
\draw[fill=white] (2,0) circle (4pt);

\draw[bend left=20, postaction=decorate] (4.75, -.25) to (2.25,-1);

\end{scope}

\node at (-1.5,-1.75) {\small{or}};

\begin{scope}

\draw[white] (0,0.5) circle (1pt);

\draw[gray!70] (2,0) -- (2,-1);
\draw[gray!70, fill=gray!70] (2,-1) circle (4pt);

\draw[dashed] (0,0) -- (5,0);

\draw (0,0) -- (.5,0);
\draw (1.5,0) -- (2.5,0);
\draw (3.5,0) -- (5,0);

\draw[gray!70] (2,0) -- (2,-1);
\draw[gray!70, fill=gray!70] (2,-1) circle (4pt);

\foreach \i in  {0,4,5} 
{
\draw[fill] (\i,0) circle (4pt);
}

\draw[fill=white] (2,0) circle (6pt);
\draw[fill=white] (2,0) circle (4pt);

\draw (4,0) -- (4,-1);
\draw[fill] (4,-1) circle (4pt);

\draw[bend left=20, postaction=decorate] (3.75, -1.25) to (2.25,-1.25);

\end{scope}

\end{tikzpicture}
\end{tabular}

\\

\hline

$\S_2$ &  \small 
$G$ has two (or more) bad left leaves 
or has 
two (or more) bad right leaves; \newline
$R(w_r) \geq L(w_{r-1})$.
&

\begin{tabular}{c}
\begin{tikzpicture}[scale=0.5,
    decoration={
      markings,
      mark=at position 1 with {\arrow[scale=1.25,black]{latex}};
    }
  ]

\draw[color=white] (0, .5) circle (1pt);

\draw[gray!70] (6,0) -- (6,-1);
\draw[gray!70, fill=gray!70] (6,-1) circle (4pt);

\draw(5,0) -- (5,-1);
\draw[fill] (5,-1) circle (4pt);

\draw[gray!70] (2,0) -- (2,-1);
\draw[gray!70, fill=gray!70] (2,-1) circle (4pt);

\draw(3,0) -- (3,-1);
\draw[fill] (3,-1) circle (4pt);

\draw[dashed] (0,0) -- (11,0);

\draw (0,0) -- (.5,0);
\draw (1.5,0) -- (3.5,0);
\draw (4.5,0) -- (6.5,0);
\draw (7.5,0) -- (9.5,0);
\draw (10.5,0) -- (11,0);

\foreach \i in  {0,2,3,5,6,11} 
{
\draw[fill] (\i,0) circle (4pt);
}

\draw[fill=white] (8,0) circle (4pt);

\draw[fill=white] (9,0) circle (6pt);
\draw[fill=white] (9,0) circle (4pt);

\draw[bend left=15, postaction=decorate] (2.9, -1.25) to (2.1,-1.25);

\draw[bend right=15, postaction=decorate] (5.1, -1.25) to (5.9,-1.25);

\node at (3.3,-1.25) {\scriptsize $x$};
\node at (4.6,-1.25) {\scriptsize $y$};

\end{tikzpicture}
\end{tabular}

\\

\hline

$\S_3$ & \small $G$ has $x \in U_2$ and another left leaf $y$;
\newline $R(w_r) > L(w_{r-1})$ &

\begin{tabular}{c}
\begin{tikzpicture}[scale=0.5,
    decoration={
      markings,
      mark=at position 1 with {\arrow[scale=1.25,black]{latex}};
    }
  ]

\draw[color=white] (2, .5) circle (1pt);

\draw[gray!70] (6,0) -- (6,-1);
\draw[gray!70, fill=gray!70] (6,-1) circle (4pt);

\draw(5,0) -- (5,-1);
\draw[fill] (5,-1) circle (4pt);

\draw[gray!70] (1,0) -- (2,0);
\draw[gray!70, fill=gray!70] (1,0) circle (4pt);

\draw(3,0) -- (3,-1);
\draw[fill] (3,-1) circle (4pt);

\draw[dashed] (2,0) -- (11,0);

\draw (2,0) -- (3.5,0);
\draw (4.5,0) -- (6.5,0);
\draw (7.5,0) -- (9.5,0);
\draw (10.5,0) -- (11,0);

\foreach \i in  {2,3,5,6,11} 
{
\draw[fill] (\i,0) circle (4pt);
}

\draw[fill=white] (8,0) circle (4pt);

\draw[fill=white] (9,0) circle (6pt);
\draw[fill=white] (9,0) circle (4pt);

\draw[bend left=20, postaction=decorate] (2.75, -1.) to (1.2,-.25);

\draw[bend right=15, postaction=decorate] (5.1, -1.25) to (5.9,-1.25);

\node at (3.35,-1.25) {\scriptsize $x$};
\node at (4.65,-1.25) {\scriptsize $y$};

\end{tikzpicture}
\end{tabular}

 \\

\hline

$\S_4$ & \small $G$ has $y\in U_{t-1}$ and another right leaf $x$;
 \newline 
$R(w_r) \geq L(w_{r-1})$ and
$L(w_r)>R(w_r)+1$ &

\begin{tabular}{c}
\begin{tikzpicture}[scale=0.5,
    decoration={
      markings,
      mark=at position 1 with {\arrow[scale=1.25,black]{latex}};
    }
  ]

\draw[color=white] (-2, .5) circle (1pt);

\draw[gray!70] (-6,0) -- (-6,-1);
\draw[gray!70, fill=gray!70] (-6,-1) circle (4pt);

\draw(-5,0) -- (-5,-1);
\draw[fill] (-5,-1) circle (4pt);

\draw[gray!70] (-1,0) -- (-2,0);
\draw[gray!70, fill=gray!70] (-1,0) circle (4pt);

\draw(-3,0) -- (-3,-1);
\draw[fill] (-3,-1) circle (4pt);

\draw[dashed] (-2,0) -- (-11,0);

\draw (-2,0) -- (-3.5,0);
\draw (-4.5,0) -- (-6.5,0);
\draw (-7.5,0) -- (-9.5,0);
\draw (-10.5,0) -- (-11,0);

\foreach \i in  {2,3,5,6,11} 
{
\draw[fill] (-\i,0) circle (4pt);
}

\draw[fill=white] (-9,0) circle (4pt);

\draw[fill=white] (-8,0) circle (6pt);
\draw[fill=white] (-8,0) circle (4pt);

\draw[bend right=20, postaction=decorate] (-2.75, -1.) to (-1.2,-.25);

\draw[bend left=15, postaction=decorate] (-5.1, -1.25) to (-5.9,-1.25);

\node at (-3.35,-1.25) {\scriptsize $y$};
\node at (-4.65,-1.25) {\scriptsize $x$};

\end{tikzpicture}
\end{tabular}

\\

\hline

$\S_5$ & \small $G$ has $y\in U_{t-1}$ and another right leaf $x$;
 \newline 
$R(w_r) \geq L(w_{r-1})$ and
$L(w_r)=R(w_r)+1$ &

\begin{tabular}{c}
\begin{tikzpicture}[scale=0.5,
    decoration={
      markings,
      mark=at position 1 with {\arrow[scale=1.25,black]{latex}};
    }
  ]

bifocal
\draw[color=white] (-2, .5) circle (1pt);

\draw(-5,0) -- (-5,-1);
\draw[fill] (-5,-1) circle (4pt);

\draw[gray!70] (-1,0) -- (-2,0);
\draw[gray!70, fill=gray!70] (-1,0) circle (4pt);

\draw(-3,0) -- (-3,-1);
\draw[fill] (-3,-1) circle (4pt);

\draw[dashed] (-2,0) -- (-10,0);

\draw (-2,0) -- (-3.5,0);
\draw (-4.5,0) -- (-5.5,0);
\draw (-6.5,0) -- (-8.5,0);
\draw (-9.5,0) -- (-10,0);

\foreach \i in  {2,3,5,10} 
{
\draw[fill] (-\i,0) circle (4pt);
}

\draw[fill=white] (-8,0) circle (4pt);

\draw[fill=white] (-7,0) circle (6pt);
\draw[fill=white] (-7,0) circle (4pt);

\draw[bend right=20, postaction=decorate] (-2.75, -1.) to (-1.2,-.25);

\node at (-3.35,-1.25) {\scriptsize $y$};
\node at (-4.65,-1.25) {\scriptsize $x$};

\end{tikzpicture}
\end{tabular}

\\
\hline

\end{tabular}
\hspace{-1in}\caption{The Phase One tree surgeries. Except for $\S_1$, the caterpillar is bifocal with best mix focus $w_r$ and another focus $w_{r-1}$. White vertices are foci and the circled vertices are best mix foci.}
\label{table:phase1}
\end{center}
\end{table}

\begin{figure}[t]
\begin{center}
\includegraphics[width=5in]{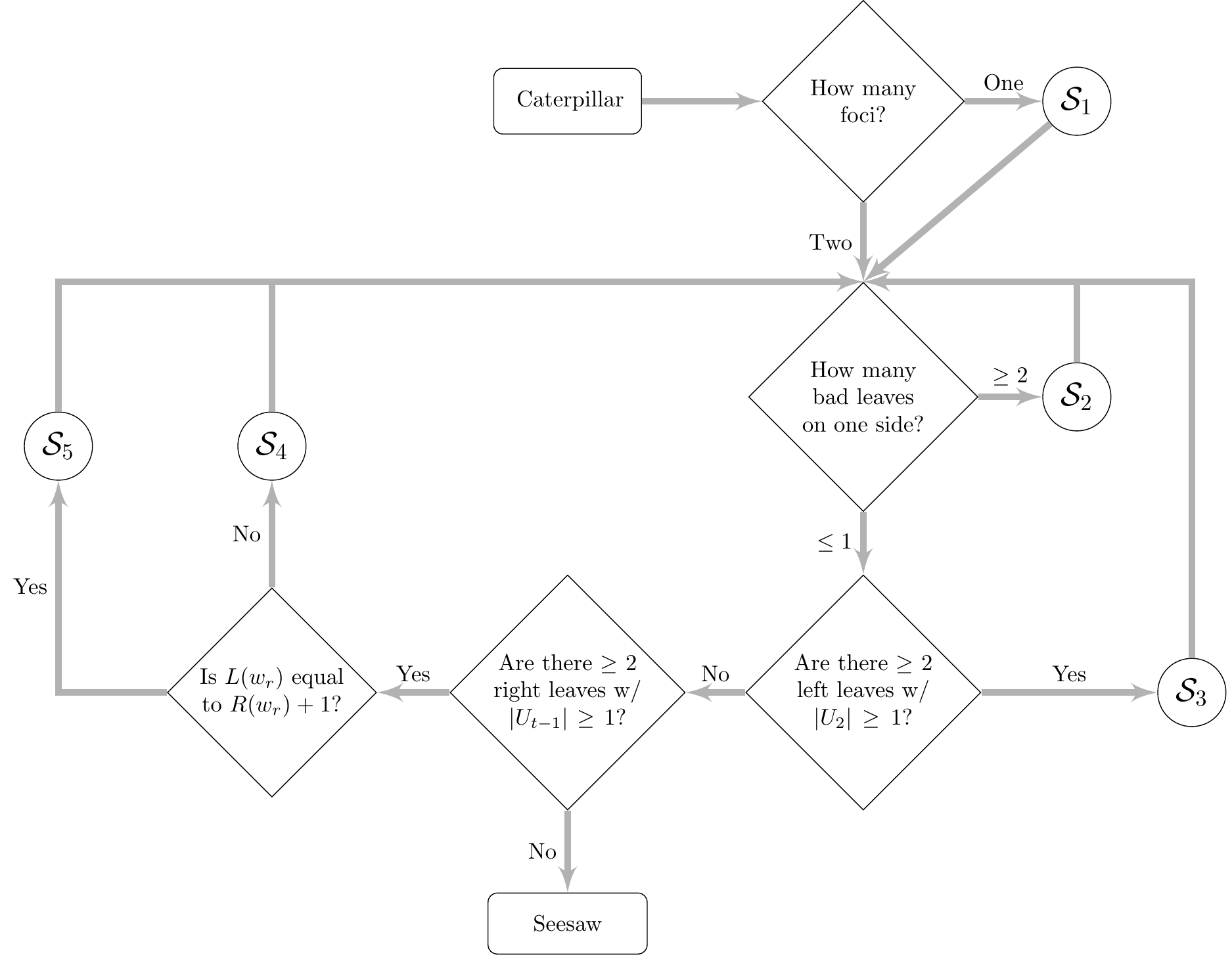}
\end{center}
\caption{Phase One of the algorithm, turning a caterpillar into a seesaw.}
\label{fig:phase1}
\end{figure}

Next, we prove that surgeries $S_1, S_2, S_3, S_4, S_5$ each result in $\Delta \Tbest \geq 0$.
If we start with a focal caterpillar $G$, we use surgery $\S_1$ to create a bifocal caterpillar $\nG = \S_1(G)$. 
Depending on the structure of $G$, we use one of two leaf transplants: $\tr{{t-1}}{r}$ or  $\str{t}{r}$.
We note that $\str{t}{r}$ is the only transplant that removes a leaf from the spine.

%
%
\begin{lemma} \label{lem:S1}
			Let $G$ be a caterpillar with a single focus $w_r$ and with the spine indexed such that $d(w_1,w_r)\geq d(w_t,w_r)$. If $U_{t-1} \neq \emptyset$, then let $S_1 = \tr{{t-1}}{r}$. If $U_{t-1} = \emptyset$ then let $\S_1 = \str{t}{r}$.
The caterpillar  $\nG = \S_1(G)$ is bifocal and $\Delta \tbm > 0$. 
\end{lemma}

 \begin{proof}
We  use Lemma \ref{lem:criteria} for our proof. The spine is indexed so that $r-1 \geq t-r$ and the unique focus means that $L_G(w_r)  = R_G(w_r).$ For both surgeries under consideration, $\Delta R(w_r) < 0$ and $\Delta L(w_r) =0$, so equation \eqref{eq:lchange} is satisfied. Next, we observe that
 $R(w_r) - L(w_{r-1}) = L(w_r) - L(w_{r-1}) = H(w_{r-1}, w_r) = \sum_{v \in V_{w_{r-1}:w_r}} \deg(v) \geq 2(r-1)-1$ by equation \eqref{eq:adjhtime}, so it  suffices to show that $\Delta L(w_{r-1}) - \Delta R(w_r) = - \Delta R(w_r) \leq 2(r-1) -1$ to verify equation \eqref{eq:newtbmfoc}.
There are two cases.
First, suppose that $U_{t-1} \neq \emptyset$, so that  $\S_1=\tr{t-1}{r}$ does not alter the spine. By equation \eqref{eq:rightright2}, we have $ -\Delta R(w_r) = 2(t-1-r) < 2(r-1)-1.$
Next, suppose that $U_{t-1} = \emptyset$, so that $\S_1=\str{t}{r}$ does alter the spine. We have
\begin{eqnarray*}
-\Delta R(w_r) &=& -H_G(w_t,w_r) + H_{\nG}(w_{t-1},w_r) \,= \, -1 -  \Delta H(w_{t-1},w_r) \\
&=& -1+2(t-1-r) \, = \, 2(t-r)-3 \,  < \, 2(r-1)-1.
\end{eqnarray*}
In either case, the conditions of Lemma \ref{lem:criteria} are satisfied, so the foci do not change. The result now follows from Corollary \ref{cor:slack} since $\Delta L(w_r)=0 \geq \Delta H(v,w_r)$ for all $v \in V$.
\end{proof}

%
%

 Next, we discuss surgery $\S_2$ which  transplants a pair of bad left-hand leaves,  moving one toward $U_2$ and one toward $U_{r-1}$, as shown in 
Table \ref{table:phase2}. We repeat $\S_2$ until the left side of the caterpillar has at most one bad leaf.

		\begin{lemma} \label{lem:S2}
			Let $G$ be a bifocal caterpillar with leaves $x\in U_i$ and $y\in U_j$ where either $2 < i \leq j < r-1$ or $r<i\leq j <t-1$. If $\S_2=\tr{i}{i-1}\wedge\tr{j}{j+1}$, then $\Delta \tbm > 0$.
		\end{lemma}
		
		\begin{proof} 
We consider the left-hand spine case $2<i\leq j<r-1$. The right-hand spine proof is analogous,   switching the roles of $i$ and $j$.
Suppose that there are  bad leaves $x \in U_i$ and $y \in U_j$. 
First, we observe that $\Delta \L (w_r) = 0 = \Delta \L(w_{r-1})$: these net hitting time changes are  $2-2=0$ as per equations \eqref{eq:leftleft2} and \eqref{eq:leftright2}. The right hand spine is unaffected, so 
$\Delta \R (w_r) = 0 = \Delta \R(w_{r-1})$.
Corollary \ref{cor:criteria} guarantees that $\nG$ also has foci $w_{r-1}$ and best mix focus $w_r$.

Having established that the foci do not change, we  show that $\tbm$ increases. We must show  that
$
				\Delta\L(w_r)=\Delta H(w_1,w_r)>\Delta H(\pi,w_r).
$
We only argue the case $i<j$, as the case $i=j$ is a straight-forward adaptation.
Equations \eqref{eq:leftleft2} and \eqref{eq:leftright2} give
$$
\begin{array}{rclcrcll}
\Delta H(x,w_r) &=& H_G(w_{i-1},w_i), & &
\Delta H(v,w_r)&=& -2 &\mbox{ for } v \in (V_i \cup \cdots \cup V_j) \backslash \{x, y \}, \\
\Delta H(y,v_r) &=& 2-H_G(w_j,w_{j+1}), &&
\Delta H(v,w_r) &=&0 & \mbox{ for all other } v.
\end{array}
$$
The only degree changes are $\Delta \deg(w_i) = -1 = \Delta \deg(w_j)$ and  		
$\Delta \deg(w_{i-1}) = 1 = \Delta \deg(w_{j+1})$. (When $i=j$, the only negative degree change is  $\Delta \deg(w_i)=-2$.)	
We use Lemma \ref{lem:slack} with $A= \{x \}$ to show that $\tbm$ does not decrease. 
The left  side of inequality \eqref{eqn:slack} simplifies to $\Delta H(x,w_r) = H_G(w_{i-1},w_i)$. It remains to show that this value is a lower bound for the right side of \eqref{eqn:slack}. 
Taking  $B= \{w_{i-1}, w_i, w_j, w_{j+1}, y \}$, we obtain
\begin{align*}
\MoveEqLeft  -H_G(w_{i-1},w_r) + \left(2\deg_{\nG}(w_i)+H_G(w_i,w_r) \right)+ \left(2\deg_{\nG}(w_j)+H_G(w_j,w_r)\right)\\
					&\qquad -H_G(w_{j+1},w_r)-(2-H_G(w_j,w_{j+1}))\\
					&= 2(\deg_G(w_i)+\deg_G(w_j)-2)-2+H_G(w_i,w_r)  -H_G(w_{i-1},w_r) \\
					&\qquad +H_G(w_j,w_r) -H_G(w_{j+1},w_r)+H_G(w_j,w_{j+1})\\
					&> 2H_G(w_j,w_{j+1})-H_G(w_{i-1},w_i) 
					>H_G(w_{i-1},w_i),
\end{align*}
where the last inequality uses equation \eqref{eq:adjhtime} twice to justify  $H_G(w_{i-1},w_i)<H_G(w_j,w_{j+1})$. 
				Thus the condition of  Lemma \ref{lem:slack} is satisfied, so $\Delta \tbm > 0$. 
\end{proof}

%
%

Once there is at most one bad left leaf, we  increase the spine length, starting with the left side.
Surgery $\S_3$ requires   at least two left leaves, one of which must be in $U_2$. 
We also require $R(w_r)>L(w_{r-1})$, meaning that $w_r$ is the unique best mix focus. If this is not the case (that is, $R(w_r) = L(w_{r-1})$), then we will take $w_{r-1}$ to be the best mix focus, reverse the labeling of the spine, and then apply $\S_4$ below.

		\begin{lemma} \label{lem:S3}
			Let $G$ be a bifocal caterpillar with $R(w_r) > L(w_{r-1})$ and with distinct vertices $x\in U_2$ and $y \in U_i$ where $2\leq i\leq r-2$. If $\S_3=\tr{2}{1}\wedge\tr{i}{i+1}$, then $\Delta \tbm \geq 0$.
		\end{lemma}
	
		\begin{proof}
First, we use Lemma \ref{lem:criteria} to show that the foci do not change. 
		This surgery extends the left-hand side of the spine. The left-pessimal hitting times to $w_{r-1}$ and $w_r$ increase by $\Delta L(w_r) = \Delta L(w_{r-1}) = 3-2=1$ (using equation  \eqref{eq:adjhtime} for the effect of $\tr{2}{1}$ and equation \eqref{eq:leftleft2} for $\tr{i}{i+1}$). The right-pessimal hitting times do not change, $\Delta R(w_r) = \Delta(w_{r-1})=0$. Therefore, equation \eqref{eq:lchange} is satisfied.
By assumption, $R(	w_r) - L(w_{r-1}) \geq 1 = \Delta\L(w_{r-1})-\Delta\R(w_r)$ so  inequality \eqref{eq:newtbmfoc} holds.
We have,  $\Delta H(v,w_r)\leq\Delta\L(w_r)$ for all $v\in V$ holds for all $v \in V$ since $H(x,w_r)$ is the only hitting time to $w_r$ that increases; all others decrease or are constant. By Corollary \ref{cor:slack}, we have $\Delta \tbm \geq 0$. 
			\end{proof}

%
%

Surgery $\S_4$ is the right-hand version of  $\S_3$. However, if
$
\L(w_r)=\R(w_r)+1
$
then applying $\S_4$ would lead to
$
\L_{\nG}(w_r)=\R_{\nG}(w_r)
$
which indicates that $\nG$ has one focus by Lemma \ref{lemma:focuscat}(a). We choose to avoid this situation,
So we require that  $\L(w_r) > \R(w_r)+1$ and handle the case $\L(w_r)=\R(w_r)+1$ with $S_5$ below.

			\begin{lemma} \label{lem:S4}
				Let $G$ be a bifocal caterpillar with
				$
					\L(w_r)>\R(w_r)+1
				$				
				 that contains distinct vertices $y  \in U_{t-1}$ and $x \in U_i$ where $r < i < t$.
				If $\S_4=\tr{t-1}{t-2}\wedge\tr{i}{i-1}$, then $\Delta \tbm > 0$.
			\end{lemma}

\begin{proof}
This surgery extends the spine on the right-hand side. Analgous to  the previous proof, this time the left-pessimal hitting times have $\Delta L(w_{r-1})= 0 = \Delta L(w_r)$ and the right-pessimal hitting times have
$\Delta R(w_{r-1}) = 1 = \Delta R(w_r)$.
Inequality \eqref{eq:lchange} holds because $L(w_r) - R(w_r) > 1 = \Delta\R(w_r)-\Delta\L(w_r),$ and
inequality \eqref{eq:newtbmfoc} holds because $\Delta\L(w_{r-1})-\Delta\R(w_r)<0$. By Lemma \ref{lem:criteria}, $\nG$ has focus $w_{r-1}$ and best mix focus $w_r$.

Next, we show that $\Delta \Tbest > 0$ using Lemma \ref{lem:slack} with $A=\{ y \}$.
We have $\Delta L(w_r) <  \Delta H(y, w_r) = \Delta R(w_r)$, while $\Delta H(v,w_r) \leq 0 = \Delta L(w_r)$ for $v \neq y$. 
The left hand side of inequality \eqref{eqn:slack} equals $\Delta H(y,w_r) = 1$. 
We take $ B= \{ w_{i-1}, w_i, w_{t-1}, w_t \}$ because the contribution from each vertex in $\overline{A} \backslash B$ is positive.
 Using equation \eqref{eq:rightright2}, this sum is
\begin{align*}
\MoveEqLeft
-H_G(w_{i-1},w_r)  + \left( 2\deg_{\nG}(w_i)+H_G(w_i,w_{r}) \right) 
+ \left(2\deg_{\nG}(w_{t-1})+H_G(w_{t-1},w_r) \right) \\
&\qquad
+ \left(-2\deg_{\nG}(w_t) -H_G(w_t,w_r) \right)\\
&=
 H_G(w_i,w_{i-1}) -H_G(w_t, w_{t-1})
+ 2 \left(  \deg_{\nG}(w_i) + \deg_{\nG}(w_{t-1}) -\deg_{\nG}(w_t) \right)
\\
& > H_G(w_i,w_{i-1}) > 1.
\end{align*}
By Lemma \ref{lem:slack}, $\Delta \Tbest > 0$.
\end{proof}

%
%

We now consider surgery $\S_5$, which is only used when $\L(w_r)=\R(w_r)+1$. This is one of the crucial moments in our algorithm:   a minor change threatens the  balance described in Lemma \ref{lemma:focuscat}.
In fact, surgery
 $\S_5=\tr{t-1}{t}$ is the first surgery that shifts the location of the foci. The resulting graph is bifocal with best mix focus $w_r$, but the second focus moves from $w_{r-1}$ to $w_{r+1}$.

\begin{lemma} \label{lem:S5}
Let $G$ be a bifocal caterpillar with $y\in U_{t-1}$ such that
$
L(w_r)= R(w_r)+1.
$
If $\S_5=\tr{t-1}{t}$, then $\Delta \tbm > 0$.
\end{lemma}
We note that during Phase One, we will also have a second right leaf (otherwise we are done with the right spine). However, our proof does not require the existence of such a leaf.
	
	\begin{proof}
Let $\nG = \S_5(G)$ and let $y \in V_{t-1}$ be the transplanted vertex. The changes in the left and right hitting times to the foci  are
$\Delta\L(w_{r-1})=0= \Delta\L(w_{r})$
and
$\Delta\R(w_{r-1})= 3 =  \Delta\R(w_r).$
Crucially, inequality \eqref{eq:lchange} is not satisfied. 
 We now verify that $w_r$ is the best mix focus of $\nG$ and that $w_{r+1}$ is the other focus. We use Lemma \ref{lemma:focuscat} (d), replacing $r$ with $r+1$ and swapping left for right. In other words we must show that $R(w_r) > L(w_r)$ and $R(w_{r+1}) \leq L(w_r)$. Observe that
$
\R_{\nG}(w_r)
=\R_G(w_r)+3
>\L_G(w_r)
=\L_{\nG}(w_r),
$
and
$$
\R_{\nG}(w_{r+1})
=\R_G(w_{r+1})+3 < \R_G(w_{r+1}) + H_{\nG} ( w_{r+1}, w_r)
= \R_G(w_r)
< \L_G(w_r)
=\L_{\nG}(w_r).
$$
Next, we  show that $\tbm(\nG)>\tbm(G)$ using Lemma \ref{lem:slack}. Note that the surgery shifts the foci, so $w_1$ is the $G$-pessimal vertex for $w_r$, while $y$ is the $G'$-pessimal vertex for $w_r$. We find that $\pdelt(w_r)=2$ because
$$
\hp_{\nG}(w_r) =  H_{\nG}(y,w_r) = 3+H_G(w_t,w_r) =2+H_G(w_1,w_r) = 2 + \hp_G(w_r).
$$
We will use Lemma \ref{lem:slack} with $A=\{y\}$.
The left side of
equation \eqref{eqn:slack} is
$
\deg_{\nG}(y)(\Delta H(y,w_r)-2)=3-2=1.
$
We take $B =  \{y, w_{t-1},w_t\}$ since $v \notin B$ means that $\Delta\deg(v)=0$  and 
$ \Delta H(v,w_r)\leq 2=\pdelt(w_r).$
The right hand side of equation \eqref{eqn:slack} is 
$$
\left( 2\deg_{\nG}(w_{t-1})+H_G(w_{t-1},w_r) \right) -H_G(w_t,w_r)
			= 2\deg_{\nG}(w_{t-1})-H_G(w_t,w_{t-1})
			> 1,
$$
		so by Lemma \ref{lem:slack}, $\Delta \tbm > 0$.
		\end{proof}

%
%

	\subsection{Phase Two: Seesaw to Twin Broom}

In this section, we discuss Phase Two, which inputs  a seesaw and outputs a twin broom. 


\begin{table}[ht]
\begin{center}
\begin{tabular}{|c|p{2.75in}|@{}c@{}|}

\hline

Surgery &  Initial Conditions & Illustration\\
\hline
$\S_6$ & \small $G$ has exactly one left leaf and  $2r-1 \leq t$ &

\begin{tabular}{c}
\begin{tikzpicture}[scale=0.5,
    decoration={
      markings,
      mark=at position 1 with {\arrow[scale=1.25,black]{latex}};
    }
  ]

\draw[color=white] (2, .5) circle (1pt);

\draw[gray!70] (1,0) -- (2,0);
\draw[gray!70, fill=gray!70] (1,0) circle (4pt);

\draw(4,0) -- (4,-1);
\draw[fill] (4,-1) circle (4pt);

\draw[dashed] (2,0) -- (9,0);

\draw (2,0) -- (2.5,0);
\draw (3.5,0) -- (5.5,0);
\draw (5.5,0) -- (7.5,0);
\draw (8.5,0) -- (9,0);

\foreach \i in  {2,4,9} 
{
\draw[fill] (\i,0) circle (4pt);
}

\draw[fill=white] (6,0) circle (4pt);

\draw[fill=white] (7,0) circle (6pt);
\draw[fill=white] (7,0) circle (4pt);

\draw[bend left=20, postaction=decorate] (3.75, -1.1) to (1.2,-.25);

\node at (4.35,-1.25) {\scriptsize $x$};

\end{tikzpicture}
\end{tabular}

\\

\hline

$\S_7$ & \small $2r-2 \geq t$; $G$ has exactly one left leaf and one right leaf &

\begin{tabular}{c}
\begin{tikzpicture}[scale=0.5,
    decoration={
      markings,
      mark=at position 1 with {\arrow[scale=1.25,black]{latex}};
    }
  ]

\draw[color=white] (2, .5) circle (1pt);

\draw[gray!70] (11,0) -- (12,0);
\draw[gray!70, fill=gray!70] (12,0) circle (4pt);

\draw(9,0) -- (9,-1);
\draw[fill] (9,-1) circle (4pt);

\draw[gray!70] (1,0) -- (2,0);
\draw[gray!70, fill=gray!70] (1,0) circle (4pt);

\draw(4,0) -- (4,-1);
\draw[fill] (4,-1) circle (4pt);

\draw[dashed] (2,0) -- (11,0);

\draw (2,0) -- (2.5,0);
\draw (3.5,0) -- (4.5,0);
\draw (5.5,0) -- (7.5,0);
\draw (8.5,0) -- (9.5,0);
\draw (10.5,0) -- (11,0);

\foreach \i in  {2,4,9,11} 
{
\draw[fill] (\i,0) circle (4pt);
}

\draw[fill=white] (6,0) circle (4pt);

\draw[fill=white] (7,0) circle (6pt);
\draw[fill=white] (7,0) circle (4pt);

\draw[bend left=20, postaction=decorate] (3.75, -1.1) to (1.2,-.25);

\draw[bend right=20, postaction=decorate] (9.25, -1.1) to (11.8,-.25);

\node at (4.35,-1.25) {\scriptsize $x$};
\node at (8.65,-1.25) {\scriptsize $y$};

\end{tikzpicture}
\end{tabular}

\\

\hline

$\S_8$ & \small $G$ has exactly one right leaf and no left leaves &

\begin{tabular}{c}
\begin{tikzpicture}[scale=0.5,
    decoration={
      markings,
      mark=at position 1 with {\arrow[scale=1.25,black]{latex}};
    }
  ]

\draw[color=white] (2, .5) circle (1pt);

\draw[gray!70] (3,0) -- (3,-1);
\draw[gray!70, fill=gray!70] (3,-1) circle (4pt);

\draw[dashed] (0,0) -- (7,0);

\draw (0,0) -- (.5,0);
\draw (1.5,0) -- (3.5,0);
\draw (4.5,0) -- (5.5,0);
\draw (6.5,0) -- (7,0);

\foreach \i in  {0,5,7} 
{
\draw[fill] (\i,0) circle (4pt);
}

\draw (5,0) -- (5,-1);
\draw[fill] (5,-1) circle (4pt);

\node at (5.35,-1.25) {\scriptsize $x$};

\draw[bend left=20, postaction=decorate] (4.75, -1.25) to (3.25,-1.25);

\draw[fill=white] (2,0) circle (4pt);

\draw[fill=white] (3,0) circle (6pt);
\draw[fill=white] (3,0) circle (4pt);

\end{tikzpicture}

\end{tabular}

\\

\hline
\end{tabular}
\hspace{-1in}\caption{The Phase Two tree surgeries.  The tree $G$ is a bifocal seesaw.}
\label{table:phase2}
\end{center}
\end{table}

\begin{figure}[t]
\begin{center}
\includegraphics[width=4.5in]{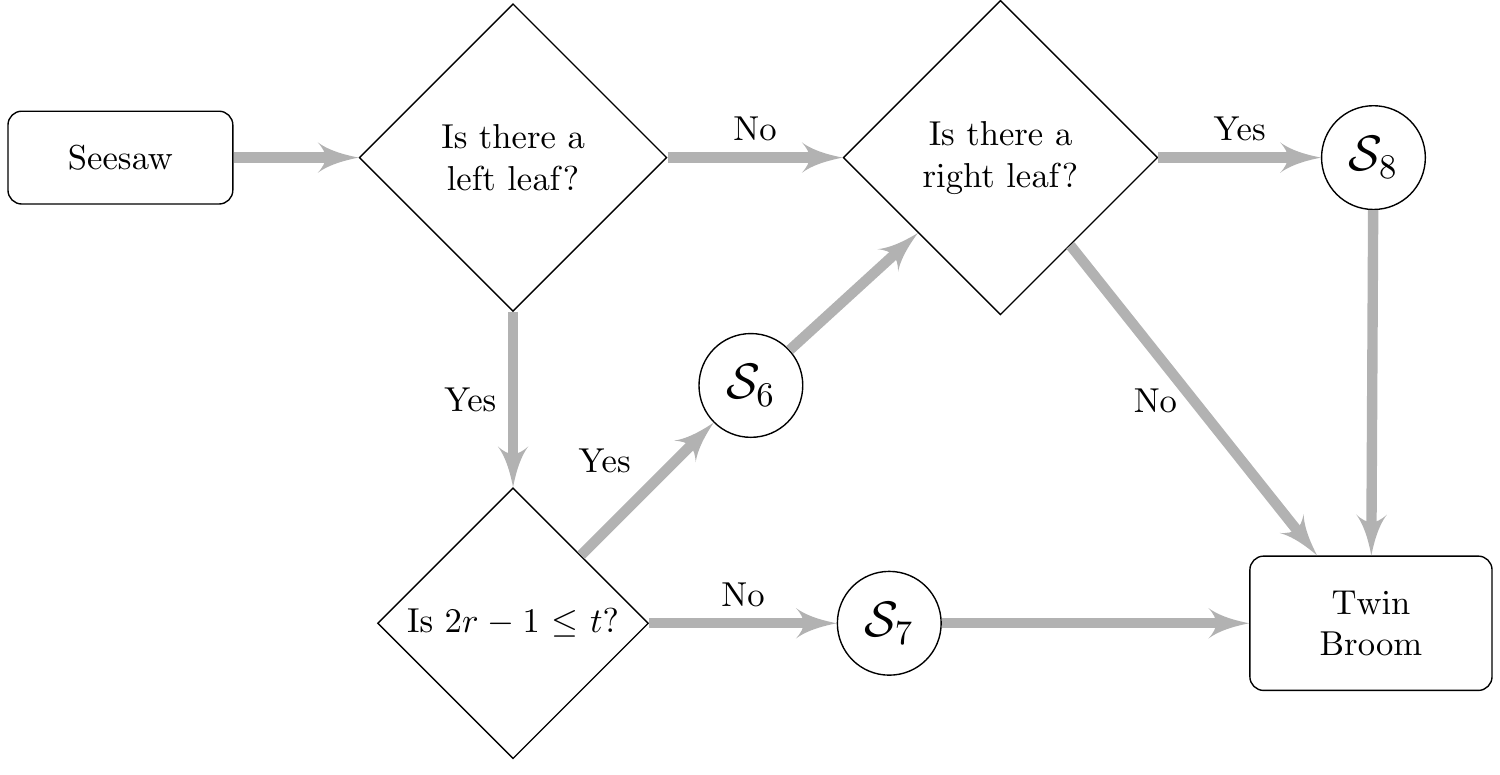}
\end{center}
\caption{Phase Two of the algorithm, turning a seesaw into a twin broom.}
\label{fig:phase2}
\end{figure}

\begin{lemma} \label{lemma:seesaw-to-twinbroom}
Let $G$ be a seesaw on $n$ vertices. Phase Two creates  a twin broom $\nG$ such that $\tbm(\nG)\geq\tbm(G)$.
\end{lemma}

\begin{proof}
If $G$ is already a twin broom then $\nG=G$. The three surgery types $\S_6, \S_7, \S_8$ employed in Phase Two are shown in
Table \ref{table:phase2} and  the workflow is shown in Figure \ref{fig:phase2}. 
The lemmas that follow show that $\Delta \Tbest \geq 0$ for all three surgeries.

Suppose that $G$ has a left leaf. If   $2r-1 \leq t$ then we use $\S_6$ to transplant the left leaf onto the end of the left spine. If $2r-2 \geq t$, or equivalently $r-2 \geq t-r$, then there must also be a right leaf because $w_r$ is the best mix focus (see Lemma \ref{lemma:focuscat}(c)). In this case, we use $S_7$ to simultaneously transplant both of these leaves to the spine. Next, if $G$ still has a right leaf (perhaps we just applied $\S_6$), then we apply $\S_8$ to move this leaf to the end of the right spine. The resulting graph is a twin broom with best mix focus $w_r$ and second focus $w_{r-1}, w_r$.
\end{proof}


First, we prove that $\S_6$ moves the final left leaf to the spine while also increasing $\Tbest$.

%
%
\begin{lemma} \label{lem:S6}
Let $G$ be a seesaw with $2r-1 \leq t$  and  exactly one left leaf  $x\in U_i$, where $2 \leq i \leq r-2$. 
If $\S_6=\tr{i}{1}$ then $\Delta \tbm > 0$.
\end{lemma} 
	
	\begin{proof}
The surgery extends the left-hand spine, so $L(w_{r-1}) = H_{\nG} (x,w_{r-1}) = H_{\nG} (x,w_1) + H_{\nG} (w_1, w_{r-1}) = 1 + H_G(w_1 + w_{r-1}) + 2(i-1)$ by equation \eqref{eq:leftleft2}. Therefore $\Delta\L(w_{r-1})=2i-1$, and likewise $\Delta\L(w_r)=2i-1,$ while   $\Delta\R(w_{r-1})=0=\Delta\R(w_r).$ 

We  show that the foci have not changed using Lemma \ref{lemma:focuscat} (d).
We have $\L_{\nG}(w_r) > \L_G(w_r) \geq R_G(w_r) = R_{\nG}(w_r)$.   Equation
\eqref{eq:htimepath2} and the assumption  $2r-1 \leq  t$  yield
 $
L_{\nG}(w_{r-1}) = (r-1)^2 \leq (t-r)^2 \leq R_{\nG}(w_r).
$
Therefore, $\nG$ has focus $w_{r-1}$ and best mix focus $w_r$.

Next, we use Lemma \ref{lem:slack} with $A=\{x\}$. The left-hand side of inequality \eqref{eqn:slack} is
$$
\Delta H(x,w_r)-\Delta\L(w_r) =(i^2-1)-(2i-1)=i^2-2i.
$$
Let $B=\{w_{r-1},w_r,\dots,w_{t-1}\}$. For all $v \in B$, we have $\deg_{\nG}(v) \geq 2$ and $\Delta H(v,w_r)=0=\Delta\deg(v)$. The right-hand side of inequality \eqref{eqn:slack} is 
$$
		\sum_{v\in B}{2(2i-1)}
		=2(2i-1)(t-r)
		\geq 2(2i-1)(r-2)
		> i(i-2 )=i^2-2i.
$$
Lemma \ref{lem:slack} gives $\Delta \tbm > 0$.
	\end{proof}

%
%
We employ surgergy $S_7$ is the particular case where $2r-2 \geq t$ and $G$ has both a left leaf and a right leaf.
This surgery  removes both leaves simultaneously. 
\begin{lemma}
Let $G$ be a seesaw with exactly one left leaf $x\in U_i$ and exactly one right leaf $y\in U_j$, such that 
 $2r-2 \geq t$.
If $\S_7=\tr{i}{1}\wedge\tr{j}{t}$, then $\Delta \tbm > 0$. 
\end{lemma}

\begin{proof}
First, we show that in fact, $2r-2 = t$, meaning that the left and right spines are the same length.
If $r-2 > t-r$ then $L(w_{r-1}) = H(w_1, w_{r-1}) > (r-2)^2 \geq (t-r+1)^2 > H(w_t, w_r) = R(w_r)$, which contradicts Lemma \ref{lemma:focuscat}(c) because $w_r$ is a best mix focus. Next, we observe that $i-1 \geq t-j$, meaning that $i$ is further from the left endpoint than $j$ is from the right endpoint. Indeed, when the left and right spines are the same length, this is necessary for $L(w_{r-1}) \leq R(w_r)$.

	By equations \eqref{eq:leftleft2} and \eqref{eq:rightright2}, we have
$\Delta\L(w_{r-1})=2i-1 = \Delta\L(w_r)$ and
$\Delta\R(w_{r-1})=2(t-j)+1=\Delta\R(w_r)$.
After the surgery, the left and right spines are equal length and leaf-free, so the conditions of Lemma \ref{lemma:focuscat}(d) 
  hold for $\nG$. In fact, both $w_{r-1}, w_r$ are now best mix foci because $L_{\nG}(w_{r-1}) = R_{\nG}(w_r)$.
 To prove that $\tbm$ increases, we use  Lemma \ref{lem:slack}. Taking $A=\{x,y\}$, the left hand side of inequality \eqref{eqn:slack}
 is
$$
 \big( i^2-1-(2i-1) \big) + \big((t+1-j)^2-1-(2i-1) \big) \leq 2  \big( i^2-1-(2i-1) \big)  = 2i^2 - 4i.
$$
For  the right hand side of inequality \eqref{eqn:slack}, we take $B=W=\{ w_1, w_2, \ldots , w_t \}$, in other words, we  ignore any central leaves. This right hand side is at least
\begin{align*}
	\MoveEqLeft  2t \, \pdelt(w_r) -2 \sum_{k=1}^t  \Delta H(w_i, w_r) \\
	& \qquad
	-H_G(w_1, w_r) + H_G(w_i, w_r) + H_G(w_j, w_r) - H_G(w_t,w_r) \\ 
& = 2t (2i-1) -4 \sum_{k=1}^{i-1}  (i-k) -4 \sum_{k=j+1}^{t}  (k-j)
	- (i-1)^2- (t-j)^2 \\ 	
& \geq 2t (2i-1) - 4 i(i-1)  - 2 ( i-1)^2  \, =\, 4(r-1)(2i-1) - 6i^2  + 8i -2 \\
& > 4(i-1)(2i-1) - 6i^2  + 8i -2 \, =\,  2i^2 - 4i +2.
\end{align*}
We have satisfied the conditions of Lemma \ref{lem:slack}, so $\Delta \tbm > 0$.
\end{proof}

%
%

Finally, we consider surgery $\S_8$, which moves the final right leaf to the spine.

	\begin{lemma}
	Let $G$ be a seesaw with exactly one right leaf  $x \in U_i$, where $r+1 \leq i \leq t-1$ and no left leaves. If $\S_8=\tr{i}{r}$, then $\Delta \tbm > 0$.
		\end{lemma} 
		
\begin{proof}		
First, we use Lemma \ref{lem:criteria}, to show that $\nG$ has focus $w_{r-1}$ and best mix focus $w_r$. 
By equation \eqref{eq:rightleft2},
$\Delta\L(w_{r-1})=0 =\Delta\L(w_r)=\Delta H(w_1,w_r)$ and
$\Delta\R(w_{r-1})=-2(i-r)=\Delta\R(w_r)$.
Inequality \eqref{eq:lchange} is clearly satisfied. Next, we verify inequality \eqref{eq:newtbmfoc}.
By equation \eqref{eq:htimepath2}, we have
$
		R_G(w_r)-L_G(w_{r-1})=(t-r)^2+2(i-r)-(r-2)^2
$
and we claim that $(t-r)^2 - (r-2)^2 >0$. Indeed, since $w_r$ is a best mix focus, Lemma \ref{lem:criteria} (c) ensures that 
$L(w_{r-1}) \leq R(w_r)$, or in other words 
$(r-2)^2 \leq (t-r)^2 + 2(r-i) < (t-r+1)^2.$ Since $r-2$ and $t-r$ are both integers, we must have $r-2 \leq t-r$. This means that $R_G(w_r)-L_G(w_{r-1}) \geq 2(i-r) = \Delta\L(w_{r-1})-\Delta\R(w_r).$ This confirms that the foci do not change.
Finally, for all $v\in V$, we have $\Delta H(v,w_r)\leq 0=\L(w_r)$. By Corollary \ref{cor:slack}, $\Delta \tbm > 0$.
\end{proof}

%
%
\subsection{Phase Three: Twin Broom to Path or Wishbone}

Phase Three converts a twin broom into either $P_n$ or $Y_n$. The three surgeries are shown in
Table \ref{table:phase3} and  the workflow is shown in Figure \ref{fig:phase3}.

\begin{lemma} \label{lemma:twinbroom-to-end}
Let $G$ be a twin broom on $n$ vertices. If $n$ is even then Phase Three turns $G$ into the path $P_n$. If $n$ is odd, then Phase Three turns $G$ into the wishbone $Y_n$. Moreover, if $n$ is even then $\Tbest(G) \leq \Tbest(P_n)$ and if $n$ is odd then
$\Tbest(G) \leq \Tbest(Y_n)$. Furthmore, equality holds if and only if $G=P_n$ for $n$ even, and $G=Y_n$ for $n$ odd.
\end{lemma}

\begin{proof}
The Phase Three surgeries  $\S_9, \S_{10}, \S_{11}$  are shown in
Table \ref{table:phase3} and  the workflow is shown in Figure \ref{fig:phase3}. 
The lemmas that follow show that $\Delta \Tbest > 0$ for all three surgeries.

Suppose that $r-2 < t-r$, or equivalently $t < 2r-2$. Since $G$ is bifocal,  there must be enough leaves in $U_{r-1}$ so that $(r-1)^2 + 2|U_{r-1}| = L(w_r) > R(w_r) = (t-r)^2$. We apply $\S_9$ until $r-2 = t-r$. At this point (and henceforth), both $w_{r-1}$ and $w_r$ are best mix foci by Lemma \ref{lemma:focuscat} (d). We may assume that $|U_{r-1}| \leq |U_{r}|$. 
If both $U_{r-1}$ and $U_r$ are nonempty, we use $\S_{10}$ to simultaneously extend the left and right spine by taking one vertex from each of these sets. We repeat this until $U_{r-1} = \emptyset$. Next, if there are multiple leaves remaining in $U_r$, we use $\S_{11}$ to extend both ends of the spine. We repeat $\S_{11}$ until there are 0 or 1 leaves left in $U_r$. At this point, we either have a path $P_n$ or a wishbone $Y_n$.
We started this phase with a twin broom, which is bifocal by definition. Therefore, if $n$ is even we have constructed a path and if $n$ is odd we  have constructed a wishbone.
\end{proof}

\begin{table}[ht]

\begin{center}
\begin{tabular}{|c|p{2.5in}|@{}c@{}|}

\hline

Surgery &  Initial Conditions & Illustration \\
\hline
$\S_9$ & \small $r-2<t-r$, which forces $U_{r-1} \neq 0$ &

\begin{tabular}{c}
\begin{tikzpicture}[scale=0.5,
    decoration={
      markings,
      mark=at position 1 with {\arrow[scale=1.25,black]{latex}};
    }
  ]

\draw[color=white] (2, .5) circle (1pt);

\draw[gray!70] (1,0) -- (2,0);
\draw[gray!70, fill=gray!70] (1,0) circle (4pt);

\draw(4,0) -- (4,-1);
\draw[fill] (4,-1) circle (4pt);

\draw[dashed] (2,0) -- (7,0);

\draw (2,0) -- (2.5,0);
\draw (3.5,0) -- (5.5,0);
\draw (6.5,0) -- (7,0);

\foreach \i in  {2,7} 
{
\draw[fill] (\i,0) circle (4pt);
}

\draw[fill=white] (4,0) circle (4pt);

\draw[fill=white] (5,0) circle (6pt);
\draw[fill=white] (5,0) circle (4pt);

\draw[bend left=20, postaction=decorate] (3.75, -1.1) to (1.2,-.25);

\node at (4,-1.5) {\scriptsize $x$};

\end{tikzpicture}
\end{tabular}

\\

\hline

$\S_{10}$ & \small $r-2=t-r$, and there are leaves $x\in U_{r-1}$ and $y\in U_r$ &
 
 \begin{tabular}{c}
\begin{tikzpicture}[scale=0.5,
    decoration={
      markings,
      mark=at position 1 with {\arrow[scale=1.25,black]{latex}};
    }
  ]

\draw[color=white] (2, .5) circle (1pt);

\draw[gray!70] (7,0) -- (8,0);
\draw[gray!70, fill=gray!70] (8,0) circle (4pt);

\draw(5,0) -- (5,-1);
\draw[fill] (5,-1) circle (4pt);

\draw[gray!70] (1,0) -- (2,0);
\draw[gray!70, fill=gray!70] (1,0) circle (4pt);

\draw(4,0) -- (4,-1);
\draw[fill] (4,-1) circle (4pt);

\draw[dashed] (2,0) -- (7,0);

\draw (2,0) -- (2.5,0);
\draw (3.5,0) -- (5.5,0);
\draw (6.5,0) -- (7,0);

\foreach \i in  {2,7} 
{
\draw[fill] (\i,0) circle (4pt);
}

\draw[fill=white] (4,0) circle (4pt);

\draw[fill=white] (5,0) circle (6pt);
\draw[fill=white] (5,0) circle (4pt);

\draw[bend left=20, postaction=decorate] (3.75, -1.1) to (1.2,-.25);

\draw[bend right=20, postaction=decorate] (5.25, -1.1) to (7.8,-.25);

\node at (4,-1.5) {\scriptsize $x$};
\node at (5,-1.5) {\scriptsize $y$};

\end{tikzpicture}
\end{tabular}

\\

\hline

$\S_{11}$ & \small $r-2=t-r$, $U_{r-1} = \emptyset$  and there are leaves $x,y\in U_r$  &

\begin{tabular}{c}
\begin{tikzpicture}[scale=0.5,
    decoration={
      markings,
      mark=at position 1 with {\arrow[scale=1.25,black]{latex}};
    }
  ]

\draw[color=white] (2, .5) circle (1pt);

\draw[gray!70] (7,0) -- (8,0);
\draw[gray!70, fill=gray!70] (8,0) circle (4pt);

\draw(5,0) -- (5.33,-1);
\draw[fill] (5.33,-1) circle (4pt);

\draw[gray!70] (1,0) -- (2,0);
\draw[gray!70, fill=gray!70] (1,0) circle (4pt);

\draw(5,0) -- (4.67,-1);
\draw[fill] (4.67,-1) circle (4pt);

\draw[dashed] (2,0) -- (7,0);

\draw (2,0) -- (2.5,0);
\draw (3.5,0) -- (5.5,0);
\draw (6.5,0) -- (7,0);

\foreach \i in  {2,7} 
{
\draw[fill] (\i,0) circle (4pt);
}

\draw[fill=white] (4,0) circle (4pt);

\draw[fill=white] (5,0) circle (6pt);
\draw[fill=white] (5,0) circle (4pt);

\draw[bend left=20, postaction=decorate] (4.40, -1.1) to (1.2,-.25);

\draw[bend right=20, postaction=decorate] (5.6, -1.1) to (7.8,-.25);

\node at (4.6,-1.5) {\scriptsize $x$};
\node at (5.4,-1.5) {\scriptsize $y$};

\end{tikzpicture}
\end{tabular}

\\

\hline

\end{tabular}
\end{center}
\caption{The  Phase Three tree surgeries. The graph $G$ is a twin broom.}
\label{table:phase3}
\end{table}


\begin{figure}[t]
\begin{center}
\includegraphics[width=5.5in]{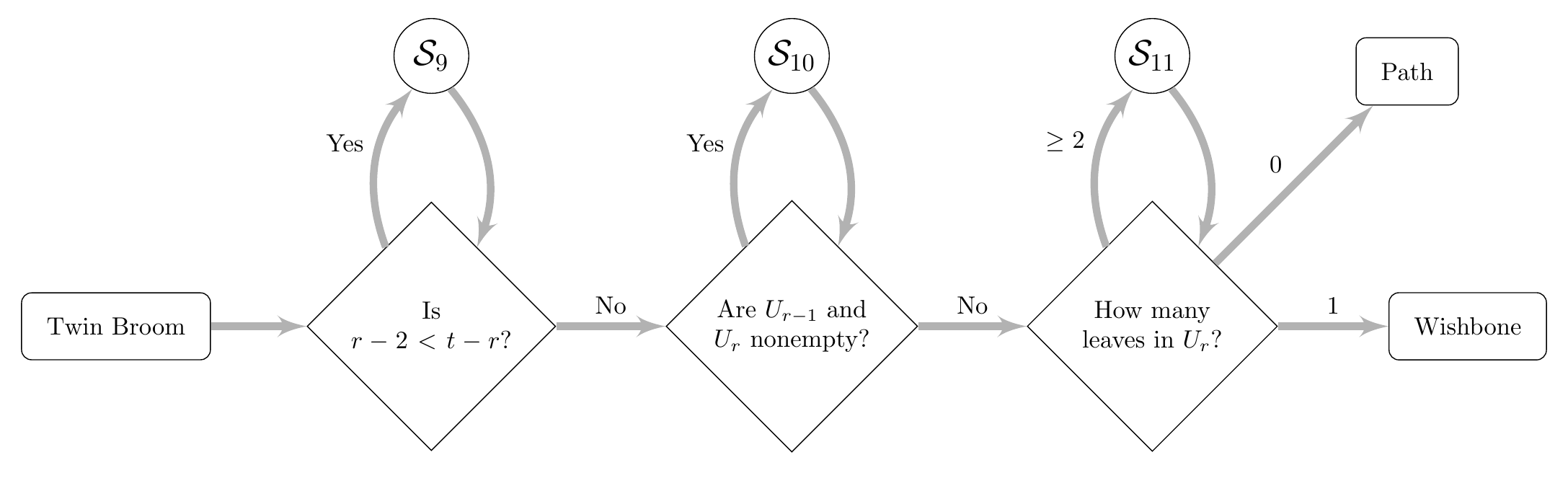}
\end{center}
\caption{Phase three of the algorithm, turning a twin broom into a path or a wishbone.}
\label{fig:phase3}
\end{figure}

%
%

We first consider $\S_9$, which we apply to equalize the lengths of the left and right spine.
	

	\begin{lemma} \label{lem:S9} 
Let $G$ be a twin broom with leaf $x\in U_{r-1}$ such that
$2r-1 \leq t.$
If $\S_9=\tr{{r-1}}{1}$, then $\Delta \Tbest > 0$.
	\end{lemma}

\begin{proof}
First, it is routine to show that the simultaneous inequalities $2r-1 \leq t$ and $\L_G(w_r) > \R_G(w_r)$  require that $\deg(w_{r-1}) \geq 4$. 
By equation \eqref{eq:leftleft2},
we have $\Delta\L(w_{r-1})=2r-3=\Delta\L(w_r)$ and
$\Delta\R(w_{r-1})=0=\Delta\R(w_r).$
We use the two criteria of Lemma \ref{lemma:focuscat} (d) to show that the foci roles do not change. Clearly
$\L_{\nG}(w_r) > \L_G(w_r) \geq \R_G (w_r) = \R_{\nG}(w_r)$.   Since $U_i=\emptyset$ for all $i\notin\{r-1,r\}$, we can use equation \eqref{eq:htimepath2} to calculate
$\L_{\nG}(w_{r-1})=(r-1)^2\leq (t-r)^2=\R_{\nG}(w_r).$

Next, we show that the best mixing time increases using Lemma \ref{lem:slack} with $A= \{ x \}$. 
 The left-hand side of inequality \eqref{eqn:slack} is
$$
\Delta H(x,w_r) - \pdelt (w_r) =( (r-1)^2-1) - (2r-3) = r^2 -4r+3.
$$
For the right hand side, we take $B=W$,
ignoring any other central leaves. We obtain
\begin{align*}
\MoveEqLeft 
(2t-1) \, \pdelt (w_r) -2 \sum_{i=1}^{r-1} \Delta H(w_i,w_r)  - H_G(w_1, w_r) + H_G(w_{r-1},w_r) \\
& = (2t-1) (2r-3) - 4 \sum_{i=1}^{r-1} (r-1-k) - (r-2)^2 \\
& \geq 2(2r-2) (2r-3)  -3 r^2 +10 r -8 \, = \, 5r^2 -10r +4, 
\end{align*}
which is clearly larger than the left hand side,  so $\Delta \Tbest > 0$.
	\end{proof}

%
%

Next, we discuss transplanting one leaf from each of $U_{r-1}$ and $U_{r}$.
		
	\begin{lemma} \label{lem:S10}
	Let $G$ be a twin broom where $t=2r-2$ and with vertices $x\in U_{r-1}$ and $y\in U_r$.
		If $\S_{10}=\tr{{r-1}}{1}\wedge\tr{r}{t}$, then $\Delta \Tbest > 0$.
	\end{lemma}

	\begin{proof}
The right-hand and left-hand spines have equal lengths before and after this surgery.
Corollary \ref{cor:criteria} holds because each  of $\Delta\L(w_{r-1})$, $\Delta\L(w_{r})$, $\Delta\R(w_{r-1})$, and $\Delta\R(w_{r})$ are equal to $2(r-2)-1= 2r-3$.
So $\nG$ has focus $w_{r-1}$ and best mix focus $w_r$.
We now show that the best mixing time increases. Setting $A=\{x,y\}$, the left hand side of inequality \eqref{eqn:slack}
is
$$\Delta H(x,w_r)+\Delta H(y,w_r)-2(2r-3)=
2 \big((r-1)^2-1 \big) - 2(2r-3) =  2r^2 -8r + 6.
$$
where we  use  equations \eqref{eq:leftleft2} and \eqref{eq:rightright2}.
As for the right-hand side, we take $B=W$, disregarding the remaining central leaves.  This right hand side is at least
\begin{align*}
\MoveEqLeft 2t \, \pdelt(w_r) - 2 \sum_{k=1}^t \Delta H(w_k,w_r) -H_G(w_1, w_r) +H(w_{r-1}, w_r) - H_G(w_t, w_r) \\
& = 2(2r-2) (2r-3) -2 \cdot 2 \sum_{k=1}^{r-1}2 (r-1-k) - 2 (r-2)^2 \, = \, 2r^2 -4.
\end{align*}
By Lemma \ref{lem:slack}, $\Delta \Tbest > 0$. 
\end{proof}

%
%

Finally, we discuss transplanting pairs of leaves from $U_r$ to the ends of the spine.
\begin{lemma} \label{lem:S10}
Let $G$ be a twin broom with $t = 2r-2$ where $U_{r-1} = \emptyset$ and  with distinct vertices $x,y \in U_r$.
If $\S_{11}=\tr{r}{1}\wedge\tr{r}{t}$, then $\Delta \Tbest > 0$.
\end{lemma}

\begin{proof}
This proof is similar to the previous one: the left and right spine lengths are equal before and after the surgery.
We have 
$\Delta\L(w_{r})=2r-1$ while  $\Delta\L(w_{r-1})=\Delta\R(w_{r-1})=\Delta \R(w_{r})=2r-3,$ where the value for $\Delta \R(w_{r-1})$ takes into account the removal of two leaves from $U_r$. We have  $	\Delta\L(w_{r-1})-\Delta\R(w_r)=0$ and $\Delta\R(w_r)-\Delta\L(w_r)=-2,$ so
 Corollary \ref{cor:criteria} ensures that $\nG$ has focus $w_{r-1}$ and best mix focus $w_r$.

 Setting $A= \{ x, y \}$, the
left hand side of inequality \eqref{eqn:slack} is
$$
 (r^2 -1) + ((r-1)^2 -1) - 2(2r-1) = 2r^2   -2r + 1.
$$
As for the right-hand side of inequality \eqref{eqn:slack}, we take $B=W$ and obtain
\begin{align*}
\MoveEqLeft  2t \, \pdelt(w_r) - 2\sum_{k=1}^t \Delta H(w_k, w_r) -H_G(w_1, w_r) - H_G(w_t, w_r) \\
&= 2(2r-2) (2r-1) - 2\sum_{k=1}^{r-1} 2(r-k) - 2 \sum_{k=r+1}^{2r-2} 2(k-r) - (r-1)^2 - (r-2)^2 \\
&=2r^2+2r-5.
\end{align*}
Thus Lemma \ref{lem:slack} ensures that $\Delta \Tbest > 0$.
\end{proof}


\section{Conclusion}

We have characterized the tree structures on $n$ vertices that minimize and maximize $\Tbest = \min_{v \in V} H(v,\pi)$. The star $S_n$ is the unique minimizing structure, but the maximization problem depends on the parity of $n$. For even $n$, the maximizing structure is the path $P_n$, and for odd $n$, it is the wishbone $Y_n$. 
It is a bit strange that the odd path is not the maximizing structure for $\Tbest$. But all is not lost: we believe that $P_n$ is the maximizing structure for a slightly different quantity.  For any graph $G$, the \emph{forget distribution} $\mu$ is the distribution achieving $\max_{v \in V} H(v,\mu) = \min_{\tau} \max_{v \in V} H(v, \tau)$. Lov\'asz and Winkler \cite{lovasz+winkler-forget} shows that $\mu$ is unique, and they give a general formula. For a tree $G$, the forget distribution is concentrated on its foci  \cite{beveridge}. When $G$ is focal, $\mu$ is a singleton distribution on the unique focus.  When $G$ is bifocal, $\mu$ is given by
$$
\mu_u = \frac{H(v',v) - H(u',v)}{2|E|} \quad \mbox{and} \quad \mu_v = \frac{H(u',u) - H(v',u)}{2|E|} 
$$
where $u,v$ are the foci of the tree. Instead of $\Tbest=\min_w H(w,\pi)$, we could instead consider the similar quantity $H(\mu, \pi)$. It is easy to see that $S_n$ minimizes $H(\mu, \pi)$ among all trees on $n$ vertices. We conjecture that $P_n$ maximizes this quantity for both even and odd $n$. 
Indeed, letting $G=P_n$ and $\nG=Y_n$, calculations show that for odd $n$, we have $H_{P_n}(\mu,\pi) = H_{Y_n}(\widetilde{\mu}, \widetilde{\pi}) + (2n-3)/(2n-2)$, and for even $n$, we have
$H_{P_n}(\mu,\pi) = H_{Y_n}(\widetilde{\mu}, \widetilde{\pi}) + 1/(2n-2)$.
Our tree surgery methods should be a fruitful line of attack, though tracking the changes in $H(\mu, \pi)$ will require a new set of lemmas.  We leave this problem for future work.


\bibliographystyle{plain}
\bibliography{bestmixref}
\end{document}